\documentclass[a4paper,11pt,oneside,reqno]{amsart}
\usepackage{silence}
\allowdisplaybreaks
\usepackage[utf8]{inputenc}
\usepackage[pdftex]{graphicx}
\usepackage{amssymb}
\usepackage{esint}
\usepackage{amsmath}
\usepackage{amsfonts}
\usepackage{amsthm}
\usepackage{bm}
\usepackage{mathrsfs}
\usepackage{mathtools}
\usepackage{comment}
\mathtoolsset{showonlyrefs=true}
\usepackage{appendix}
\usepackage{enumerate}
\usepackage[a4paper, total={6.6in, 9.7in}]{geometry}
\usepackage[stretch=30,shrink=30]{microtype}
\usepackage{stmaryrd}
\usepackage[normalem]{ulem} 
\usepackage{xcolor}
\definecolor{antiquefuchsia}{rgb}{0.57, 0.36, 0.51}
\definecolor{azure}{rgb}{0.0, 0.5, 1.0}

\usepackage[backref=page, colorlinks = true, linkcolor = black, citecolor = antiquefuchsia]{hyperref}

\renewcommand*{\backref}[1]{}
\renewcommand*{\backrefalt}[4]{%
    \ifcase #1 (Not cited.)%
    \or        (Cited on page~#2.)%
    \else      (Cited on pages~#2.)%
    \fi}

\makeatletter
\def\th@plain{%
	\thm@notefont{}
	\itshape 
}
\def\th@definition{%
	\thm@notefont{}
	\normalfont 
}
\makeatother

\numberwithin{equation}{section}

\newtheorem{theorem}{Theorem}[section]
\newtheorem{lemma}[theorem]{Lemma}
\newtheorem{proposition}[theorem]{Proposition}
\newtheorem{corollary}[theorem]{Corollary}

\theoremstyle{definition}

\newtheorem{remark}[theorem]{Remark}

\renewcommand{\c}{\mathbb{C}}

\newcommand{\cp}{\mathbb{CP}}

\title[Topology, Morse Index, and Area of Minimal Surfaces in Three-Manifolds]{Topology, Morse Index, and Area of Minimal Surfaces in Three-Manifolds}

\author[Riccardo Caniato]{Riccardo Caniato}
\address{Warwick Mathematics Institute, Zeeman Building, University of Warwick, Coventry, CV4 7AL, United Kingdom}
\email{riccardo.caniato@warwick.ac.uk}
\date{\today} 
\subjclass[2020]{Primary 53A10; Secondary 53C42, 58E12, 58J50}

\begin{document}
\begin{abstract}
    We prove that, in every closed Riemannian three-manifold
    $(M^3,\bar g)$, there exists a constant $C=C(M,\bar g)>0$
    such that every closed smoothly embedded minimal surface
    $\Sigma\subset M$ satisfies
    \begin{align*}
        b_1(\Sigma;\mathbb{Z}_2)
        \le C\bigl(\operatorname{Ind}(\Sigma)
        +\operatorname{Area}(\Sigma)\bigr),
    \end{align*}
    establishing Song's additive genus--index--area conjecture.
    For closed connected orientable two-sided minimal immersions,
    we prove the universal inequality
    $\gamma(\Sigma)\le8\operatorname{Ind}(\Sigma)$ in positive
    Ricci curvature and the estimate
    $\gamma(\Sigma)\le C(M,\bar g)\operatorname{Ind}(\Sigma)$
    in positive scalar curvature. In the latter case, we also
    bound both area and total curvature by
    $C(M,\bar g)(1+\operatorname{Ind}(\Sigma))$.
\end{abstract}
\maketitle
\tableofcontents
\section{Introduction}
\label{Section: Introduction}

A classical theme in differential geometry is to understand to what extent geometric information constrains topology. Some of the foundational results of the subject are manifestations of this principle: the Cartan--Hadamard theorem derives strong global topological consequences from a sign condition on the sectional curvature, while the Gauss--Bonnet theorem recovers the Euler characteristic of a closed surface from its total Gaussian curvature. For minimal surfaces, it is natural to seek analogous relations in which topology is controlled by geometric and variational quantities. Among the most fundamental of these are the area and the Morse index, which measure, respectively, the size of the surface and its instability. In this paper, we prove that the first Betti number of every closed embedded minimal surface in a closed Riemannian three-manifold is bounded by an ambient-dependent constant times the sum of its Morse index and area. For closed connected orientable two-sided minimal immersions, we prove linear genus--index bounds, with a universal constant in positive Ricci curvature and an ambient-dependent constant in positive scalar curvature. In the latter case, we also establish that both the area and the total curvature of the immersion are bounded linearly by the index plus one, with constants depending only on the ambient metric.

An immersion $\iota:\Sigma\to(M^3,\bar g)$ is \emph{minimal} if its mean curvature vector vanishes, or equivalently if it is a critical point of the area functional under compactly supported variations. Its \emph{Morse index} $\operatorname{Ind}(\Sigma)$ is the maximal dimension of a space of smooth normal variations on which the second variation of area is negative definite. A minimal surface is \emph{stable} if its index is zero. For a closed surface the index is finite; it counts strictly negative eigenvalues of the normal Jacobi operator, with multiplicity, and does not include nullity. An immersion is \emph{two-sided} if its normal line bundle is trivial. In this case, the choice of a global unit normal $\nu$ identifies normal variations with functions, and the second variation is given by
\begin{align*}
    \delta^2\operatorname{Area}_\Sigma(f\nu,f\nu)=\int_\Sigma\bigl(|\nabla f|^2-\bigl(|A_\Sigma|^2+\operatorname{Ric}_M(\nu,\nu)\bigr)f^2\bigr)\,\omega_\Sigma \qquad\forall\, f\in C^{\infty}(\Sigma,\mathbb{R}).
\end{align*}
Here, $A_\Sigma$ is the second fundamental form and $\omega_\Sigma$ denotes the induced area measure. If the immersion is one-sided, i.e. its normal bundle is nontrivial, the index is defined directly on normal sections. Throughout, we say that a manifold is \textit{closed} if it is compact and has no boundary. All geometric quantities associated with an immersion are computed on its domain equipped with the induced metric $h:=\iota^*\bar g$. We write $\gamma(\Sigma)$ for the genus of a connected orientable surface and
$b_1(\Sigma;\mathbb{Z}_2):=\dim H_1(\Sigma;\mathbb{Z}_2)$.

The interaction between stability, curvature, and topology is already central to the work of Schoen--Yau and Fischer-Colbrie--Schoen \cite{SY79,FCS80}. For instance, testing the stability inequality with the constant function $1$ and using the Gauss equation shows that a closed orientable two-sided stable minimal surface in a three-manifold of positive scalar curvature must have genus zero. If the ambient has only nonnegative scalar curvature, the only additional possibility is a totally geodesic flat torus. These conclusions illustrate the strength of the stability inequality, but also the difficulty at positive index: the inequality then holds only on a subspace of finite codimension, and constant functions need not belong to that subspace.

Several approaches relate the index to topology in particular ambient geometries. Ros used harmonic one-forms to construct normal test functions and obtain index bounds for minimal surfaces in flat three-tori and in the Euclidean three-space \cite{Ros06}. For complete minimal surfaces of finite index in $\mathbb{R}^3$, Chodosh--M\'aximo established lower bounds for the index in terms of the genus and the number of ends, and subsequently refined these estimates to account for the multiplicities of the ends \cite{CM16,CM23}. In the embedded Euclidean setting, Meeks--P\'erez--Ros proved that the number of ends of a complete minimal surface of finite topology is bounded in terms of its genus alone; when the index is finite, they also bounded it from above by a function of the genus \cite{MPR19}. In higher dimensions, Savo related the Jacobi spectrum of a minimal hypersurface of the sphere to the Hodge Laplacian on one-forms \cite{Savo10}. Building on these methods, Ambrozio--Carlotto--Sharp proved linear index--Betti-number estimates under an extrinsic curvature condition, including all compact rank-one symmetric spaces \cite{ACS}. Further index-topology estimates for complete embedded minimal hypersurfaces in $\mathbb{R}^4$ were recently obtained by Chodosh--Gianocca \cite[Theorem~10.2]{CG26}. These results support the conjecture of Schoen and Marques--Neves that, in a fixed closed manifold of positive Ricci curvature, the first Betti number over $\mathbb{R}$ of a closed embedded minimal hypersurface is bounded above by a constant times its Morse index; see \cite[Section~8]{NevesICM14}, \cite[Section~4.1]{MarquesICM14}, and \cite[Introduction]{ACS}.

There is also a complementary estimate in the opposite direction. Ejiri--Micallef compared the second variations of area and energy and proved that every closed connected orientable minimal immersion into a fixed closed Riemannian manifold satisfies
\begin{align}\label{Equation: Ejiri Micallef index estimate}
    \operatorname{Ind}(\Sigma)\le C\bigl(\gamma(\Sigma)+\operatorname{Area}(\Sigma)\bigr),
\end{align}
where $C$ depends only on the ambient metric \cite[Theorem 4.3]{EM}. In particular, this upper bound requires no control of the second fundamental form. Together with a converse estimate, it would identify genus plus area and index plus area as comparable measures of the complexity of a minimal surface.

Compactness theory provides another connection between these quantities. Choi--Schoen proved smooth compactness for closed embedded minimal surfaces of fixed topological type in a closed three-manifold of positive Ricci curvature \cite{CS85}. In the direction relevant here, the regularity theory of Schoen--Simon for stable minimal hypersurfaces \cite{SS81} underlies Sharp's compactness theorem for embedded minimal hypersurfaces with bounded area and index in ambient dimensions three through seven \cite{Sharp17}. After passing to a subsequence, the convergence is smooth and graphical, possibly with multiplicity, away from finitely many points where curvature may concentrate. Chodosh--Ketover--M\'aximo described the degenerations of bounded-index hypersurfaces \cite{CKM17}, while Buzano--Sharp established a quantization of the total curvature in terms of Euclidean minimal hypersurfaces arising as bubbles \cite{BS18}. Among the consequences of these works is finiteness of the possible diffeomorphism types under simultaneous area and index bounds. Related questions concerning the behaviour of the index under bubbling were studied by Da Lio--Gianocca--Rivi\`ere for conformally invariant variational problems on surfaces in \cite{DGR25}, where they established upper semicontinuity of the Morse index plus nullity, accounting for the contributions of both the limiting map and the bubbles.

Quantitative estimates must additionally track the dependence on both bounds. M\'aximo proved that, for closed embedded minimal surfaces in a closed three-manifold and a prescribed upper bound $I_0$ for the index, the genus is bounded by the area times a constant depending on the ambient metric and $I_0$ \cite{Maximo23}. Song subsequently proved that, in ambient dimensions three through seven, the total Betti number of a closed embedded minimal hypersurface of area at most $A$ is bounded by $C_A(1+\operatorname{Ind}(\Sigma))$ \cite{Song}. In ambient dimension three, his constant depends linearly on $A$, yielding
\begin{align}\label{Equation: Song multiplicative topology estimate}
    b_1(\Sigma;\mathbb{Z}_2)\le C\operatorname{Area}(\Sigma)\bigl(1+\operatorname{Ind}(\Sigma)\bigr),
\end{align}
with $C$ determined solely by the ambient manifold. Thus, the dependence on each quantity separately is linear, whereas their joint dependence contains a product. Song conjectured that this product could be replaced by an additive dependence, or equivalently, in view of \eqref{Equation: Ejiri Micallef index estimate}, that genus plus area and index plus area should be comparable \cite[Section 6.1, Conjecture $\mathbf{C}_1(1)$]{Song}. Cordero-Misteli and Franz \cite{CF24} extended Song's covering method to free boundary minimal hypersurfaces, obtaining
corresponding bounds for the topology of the hypersurface and its boundary, with a constant depending linearly on the area bound in ambient dimension three.

The distinction is particularly relevant to the Almgren--Pitts min--max theory \cite{Pitts81}, which constructs minimal hypersurfaces from nontrivial families of cycles. The multiparameter constructions of Marques--Neves \cite{MN17}, the density theorem of Irie--Marques--Neves \cite{IMN18}, and Song's resolution of Yau's conjecture \cite{SongYau23} show that closed minimal hypersurfaces occur abundantly. The corresponding existence theory gives quantitative information on area and index, and one would like to deduce equally precise information on topology. The Weyl law of Liokumovich--Marques--Neves \cite{LMN18}, the multiplicity-one theorem of Zhou \cite{Zhou20}, and the index characterization of Marques--Neves \cite{MN21} provide the relevant estimates for the volume spectrum. In dimension three, the Allen--Cahn approach of Chodosh--Mantoulidis gives a related realization of the Morse-theoretic picture \cite{CMant20}. Area alone cannot bound topology: Chu--Stern constructed, for generic metrics on closed three-manifolds, sequences of embedded minimal surfaces with bounded area and unbounded genus \cite{ChuStern25}. Chu also proved that every Riemannian three-sphere of positive Ricci curvature contains embedded minimal surfaces of every genus with a uniform upper bound for their areas \cite{Chu25}. Conversely, fixed ambient manifolds can contain stable embedded minimal surfaces of unbounded genus and area; see \cite[Remark 3]{Maximo23}. Both quantities are therefore needed for a general estimate without curvature assumptions.
\subsection{Main results}
Our first theorem establishes the additive estimate in arbitrary closed Riemannian three-manifolds, including nonorientable and one-sided surfaces.
\begin{theorem}[Topology, index, and area]
\label{Theorem: Topology index and area}
    Let $(M^3,\bar g)$ be a closed Riemannian manifold. There
    exists $C=C(M,\bar g)>0$ such that every closed smoothly embedded
    minimal surface $\Sigma\subset M$, possibly disconnected,
    satisfies
    \begin{align}\label{Equation: Main topology estimate}
        b_1(\Sigma;\mathbb{Z}_2)
        \le C\bigl(\operatorname{Ind}(\Sigma)
        +\operatorname{Area}(\Sigma)\bigr).
    \end{align}
\end{theorem}
Note that, for a connected orientable surface, $b_1(\Sigma;\mathbb{Z}_2)=2\gamma(\Sigma)$, while for a connected nonorientable surface the Betti number equals its nonorientable genus. In particular, Theorem~\ref{Theorem: Topology index and area} and
\eqref{Equation: Ejiri Micallef index estimate} together imply that
\begin{align}\label{Equation: Additive genus index area comparison}
    C^{-1}\bigl(\gamma(\Sigma)+\operatorname{Area}(\Sigma)\bigr)\le\operatorname{Ind}(\Sigma)+\operatorname{Area}(\Sigma)\le C\bigl(\gamma(\Sigma)+\operatorname{Area}(\Sigma)\bigr)
\end{align}
for every closed connected orientable embedded minimal surface in
$(M^3,\bar g)$, after enlarging $C$. Thus, genus plus
area and index plus area are comparable with constants determined
only by the ambient metric.

As mentioned above, the additive dependence established in Theorem~\ref{Theorem: Topology index and area} gives a linear topological bound for sequences with prescribed index and area growth. In particular, suppose that $\bar g$ is a bumpy metric on a closed connected orientable
three-manifold $M$; that is, no closed immersed minimal hypersurface admits a nonzero normal Jacobi field. By the multiplicity-one theorem of Zhou and
the index characterization of Marques--Neves
\cite{Zhou20,MN21}, for every positive integer $p$ the $p$-width
$\omega_p(M,\bar g)$ is realized by a closed embedded two-sided
minimal surface $\Sigma_p$, possibly disconnected, such that
\begin{align*}
    \operatorname{Ind}(\Sigma_p)=p,
    \qquad
    \operatorname{Area}(\Sigma_p)=\omega_p(M,\bar g).
\end{align*}
The Weyl law \cite{LMN18} gives
$\operatorname{Area}(\Sigma_p)=O(p^{1/3})$. Thus, Theorem~\ref{Theorem: Topology index and area} yields
$b_1(\Sigma_p;\mathbb{Z}_2)\le Cp$.
Conversely, all components of $\Sigma_p$ are orientable, and summing
\eqref{Equation: Ejiri Micallef index estimate} over them gives
\begin{align*}
    p\le C\bigl(b_1(\Sigma_p;\mathbb{Z}_2)
    +\operatorname{Area}(\Sigma_p)\bigr).
\end{align*}
Since the area is $o(p)$, there exist $c,C>0$ such that
\begin{align}\label{Equation: Linear topology of volume spectrum surfaces}
    cp\le b_1(\Sigma_p;\mathbb{Z}_2)\le Cp
\end{align}
for all sufficiently large $p$. In particular, the sum of the genera of the components grows linearly with $p$. The linear lower bound was already observed in \cite{Song}; see also \cite[Corollary~1.10]{CMant20} for the Allen--Cahn setting. The matching upper bound established here gives the linear growth anticipated by Marques and Neves \cite[Section~4.1]{MarquesICM14}, \cite[Section~8]{NevesICM14}, and subsequently conjectured in \cite{Song}, improving the upper bound of order $p^{4/3}$ obtained from \eqref{Equation: Song multiplicative topology estimate}.

For orientable two-sided immersions, the estimate can be made
explicit in terms of a lower bound for the ambient sectional
curvature. This result does not require embeddedness or compactness
of the ambient manifold.

\begin{theorem}[An explicit estimate for two-sided immersions]
\label{Theorem: Explicit two sided estimate}
    Let $\iota:\Sigma\to(M^3,\bar g)$ be a closed connected
    orientable two-sided minimal immersion. Suppose that all ambient
    sectional curvatures along $\iota(\Sigma)$ are at least
    $-\kappa$, for some $\kappa\ge0$. Then
    \begin{align}\label{Equation: Main explicit estimate}
        \gamma(\Sigma)
        \le8\operatorname{Ind}(\Sigma)+1
        +\frac{7\kappa}{4\pi}\operatorname{Area}(\Sigma).
    \end{align}
    If $\operatorname{Ind}(\Sigma)=0$, the stronger estimate
    \begin{align}\label{Equation: Stable explicit estimate}
        \gamma(\Sigma)
        \le1+\frac{\kappa}{2\pi}\operatorname{Area}(\Sigma)
    \end{align}
    holds.
\end{theorem}

The same argument yields a bound independent of area when the ambient
Ricci curvature is nonnegative. More precisely, for every closed
connected orientable two-sided minimal immersion, we prove
\begin{align}\label{Equation: Main Ricci genus estimates}
    \gamma(\Sigma)&\le8\operatorname{Ind}(\Sigma)+1\ \qquad \text{ if }\operatorname{Ric}_M\ge0,\notag\\
    \gamma(\Sigma)&\le8\operatorname{Ind}(\Sigma)\qquad\qquad \text{ if }\operatorname{Ric}_M>0,
\end{align}
where the curvature assumptions are imposed along the immersion;
see Corollary~\ref{Corollary: Ricci curvature and genus}.
In particular, positive Ricci curvature gives
\begin{align*}
    b_1(\Sigma;\mathbb{R})
    =2\gamma(\Sigma)\le16\operatorname{Ind}(\Sigma).
\end{align*}
This proves the orientable two-sided surface case of the
Marques--Neves--Schoen conjecture described above. The constant is
independent of the ambient metric, and the immersion need not be
embedded. For comparison, the sharper genus bound
$\gamma(\Sigma)\le3$ for index-one surfaces under nonnegative Ricci
curvature was already known \cite{Ros06}; Ros subsequently constructed
a genus-three index-one embedded minimal surface in a metric of
positive sectional curvature on $\mathbb{RP}^3$ \cite{Ros25}.

For orientable two-sided embedded surfaces, one also obtains the converse comparison in
a fixed closed orientable three-manifold of positive Ricci curvature.
Indeed, choose $\rho>0$ such that
$\operatorname{Ric}_M\ge\rho\bar g$. The first-eigenvalue estimate
of Choi--Wang \cite{CW83} and the Yang--Yau inequality \cite{YangYau}
give
\begin{align*}
    \frac{\rho}{2}\operatorname{Area}(\Sigma)
    \le\lambda_1(\Sigma)\operatorname{Area}(\Sigma)
    \le8\pi\bigl(\gamma(\Sigma)+1\bigr).
\end{align*}
Together with \eqref{Equation: Ejiri Micallef index estimate}, this
implies
$\operatorname{Ind}(\Sigma)\le C(\gamma(\Sigma)+1)$.
On the other hand, positive Ricci curvature implies
$\operatorname{Ind}(\Sigma)\ge1$, so
\eqref{Equation: Main Ricci genus estimates} gives
$\gamma(\Sigma)+1\le9\operatorname{Ind}(\Sigma)$.
Consequently,
\begin{align}\label{Equation: Positive Ricci genus index comparison}
    C^{-1}\bigl(\gamma(\Sigma)+1\bigr)
    \le\operatorname{Ind}(\Sigma)
    \le C\bigl(\gamma(\Sigma)+1\bigr)
\end{align}
for closed connected orientable two-sided embedded minimal surfaces
in this setting. This establishes
\cite[Section 6.1, Conjecture $\mathbf{C}_1(3)$]{Song}
in that class.

Positive scalar curvature also gives bounds in terms of the index alone,
with constants depending on the ambient metric. For connected closed embedded surfaces, Chodosh--Ketover--M\'aximo
proved that a prescribed index bound in a fixed closed
three-manifold of positive scalar curvature implies bounds
for area and genus \cite[Theorem~1.3]{CKM17}.
The following theorem gives linear dependence on the index and
applies to orientable two-sided immersions.

\begin{theorem}[Positive ambient scalar curvature]
\label{Theorem: Positive scalar curvature}
    Let $(M^3,\bar g)$ be a closed Riemannian manifold with positive
    scalar curvature. There exists $C=C(M,\bar g)>0$ such that every
    closed connected orientable two-sided minimal immersion
    $\iota:\Sigma\to M$ satisfies
    \begin{align}\label{Equation: Positive scalar genus estimate}
        \gamma(\Sigma)\le C\operatorname{Ind}(\Sigma)
    \end{align}
    and
    \begin{align}\label{Equation: Positive scalar area and curvature estimate}
        \operatorname{Area}(\Sigma)
        +\int_\Sigma |A_\Sigma|^2\,\omega_\Sigma
        \le C\bigl(1+\operatorname{Ind}(\Sigma)\bigr).
    \end{align}
\end{theorem}

Together with \eqref{Equation: Ejiri Micallef index estimate} and
the uniform positive lower area bound furnished by the monotonicity
formula, Theorem~\ref{Theorem: Positive scalar curvature} gives
\begin{align}\label{Equation: Positive scalar genus area index comparison}
    C^{-1}\bigl(1+\operatorname{Ind}(\Sigma)\bigr)\le\gamma(\Sigma)+\operatorname{Area}(\Sigma)\le C\bigl(1+\operatorname{Ind}(\Sigma)\bigr)
\end{align}
for every immersion in the theorem, after enlarging $C$; see
Lemma~\ref{Lemma: Uniform minimum area}. This proves
\cite[Section 6.1, Conjecture $\mathbf{C}_1(2)$]{Song} for connected
orientable two-sided embedded surfaces and extends the comparison
to immersions. The same comparison holds with
$\int_\Sigma|A_\Sigma|^2\,\omega_\Sigma$ added to the middle expression.
Thus positive scalar curvature suffices for an ambient-dependent
genus--index bound, while the positive-Ricci estimate retains the
universal coefficient $8$.

The additive constant in
\eqref{Equation: Positive scalar area and curvature estimate} is
necessary: in the product of a round $\mathbb{S}^2$ and
$\mathbb{S}^1$, every slice $\mathbb{S}^2\times\{t\}$ is a stable
minimal sphere of positive area. Arbitrarily many disjoint slices
also show why connectedness is needed. More generally, summing
the theorem over the connected components $\Sigma_j$ gives
\begin{align*}
    \sum_j\gamma(\Sigma_j)&\le C\operatorname{Ind}(\Sigma),\\
    \operatorname{Area}(\Sigma)
    +\int_\Sigma|A_\Sigma|^2\,\omega_\Sigma
    &\le C\bigl(b_0(\Sigma)+\operatorname{Ind}(\Sigma)\bigr),
\end{align*}
where $b_0(\Sigma)$ denotes the number of connected components.

\subsection{A scalar spectral inequality and the proof strategy}

The analytic argument belongs to the tradition of using conformal
and holomorphic maps as spectral test functions. Hersch's inequality
for the two-sphere \cite{Hersch} and the Yang--Yau estimate for
higher-genus surfaces \cite{YangYau} relate the first eigenvalue to
the energy of a suitably balanced map. Li--Yau developed the
conformal-volume approach \cite{LY}, while
Bourguignon--Li--Yau used full holomorphic immersions into projective space to obtain
first-eigenvalue estimates on K\"ahler manifolds \cite{BLY}.
Arezzo--Ghigi--Loi extended this construction using globally
generated vector bundles with stable Gieseker point \cite{AGL}.
Biliotti--Ghigi established a balancing theorem for measures that
give zero mass to projective hyperplane sections \cite{BG13}.
These results provide the projective averaging and energy identities
used below.

Related refinements of the projective construction for surfaces
were developed by Ros \cite{RosGenusThree,RosLargeGenus} and
Karpukhin--Vinokurov \cite{KV22}. Using holomorphic curves in
projective spaces and perturbations of the associated Hermitian
projection maps, they obtained improved bounds for the first
Laplace eigenvalue. In particular, the freedom to choose the
projective dimension also plays a role in \cite{KV22}.

Higher-eigenvalue estimates also form part of this picture.
Korevaar obtained bounds for Laplace eigenvalues on surfaces in
terms of genus and eigenvalue number \cite{Korevaar}.
Hassannezhad established estimates with additive dependence on
genus and eigenvalue number, and treated Schr\"odinger operators
with signed potentials \cite{Has11,Hass13}.
Kokarev obtained higher-eigenvalue bounds controlled by the degree
of a holomorphic map into projective space \cite{Kok20}.
Grigor'yan--Nadirashvili--Sire proved lower bounds for the number
of negative eigenvalues of Schr\"odinger operators with signed
potentials in terms of the integral of the potential, with
constants depending only on the genus in dimension two
\cite{GNS16}.

Our main analytic contribution is a scalar spectral inequality
that quantifies the effect of the negative spectral subspace on
balanced projective test functions. The estimate retains the
explicit genus coefficient $4\pi$, while dividing the index
contribution by an auxiliary dimension parameter. This quantitative
form allows the genus term to be absorbed in the Gauss--Bonnet
argument for curvature potentials, or to cancel exactly in the
positive-scalar-curvature case.

Let $(\Sigma,h)$ be a closed connected
oriented Riemannian surface, and denote its area form by
$\omega_\Sigma$. For $V\in C^\infty(\Sigma,\mathbb{R})$, set
\begin{align}\label{Equation: Scalar quadratic form}
    Q_V(f):=\int_\Sigma\bigl(|\nabla f|_h^2-Vf^2\bigr)\,
    \omega_\Sigma,
    \qquad
    I_V:=\operatorname{Ind}(Q_V),
\end{align}
and write
\begin{align}\label{Equation: Potential positive negative parts}
    V_+:=\max\{V,0\},
    \qquad V_-:=\max\{-V,0\}.
\end{align}
Here $I_V$ is computed on real-valued functions, or equivalently is
the number of strictly negative eigenvalues, counted with
multiplicity, of $-\Delta_h-V$, with
$\Delta_h=\operatorname{div}_h\nabla$.
\begin{theorem}[A scalar spectral inequality]
\label{Theorem: Scalar spectral inequality}
    Let $(\Sigma,h)$ be a closed connected oriented Riemannian
    surface of genus $\gamma$, and let
    $V\in C^\infty(\Sigma,\mathbb{R})$. For every integer $N\ge2$,
    \begin{align}\label{Equation: Main scalar spectral inequality}
        \int_\Sigma V\,\omega_\Sigma
        \le4\pi(\gamma+N-1)
        +\frac{I_V}{N}\int_\Sigma V_+\,\omega_\Sigma.
    \end{align}
\end{theorem}
Theorem~\ref{Theorem: Scalar spectral inequality} is of independent
interest as a general estimate for Schr\"odinger operators on closed
surfaces. It provides a tool for deriving Morse index bounds in
geometric variational problems whose second variation has scalar
Schr\"odinger form. The freedom to choose $N$ allows the estimate
to be optimized according to the geometric information encoded
in the potential.

We briefly describe the proof of Theorem~\ref{Theorem: Scalar spectral inequality}. Recall that its metric and orientation give $\Sigma$ a complex structure. For every $N\ge2$, the Riemann--Roch theorem and dimension estimates in the Picard variety provide a holomorphic map
\begin{align*}
    F:\Sigma\rightarrow\cp^{N-1},
    \qquad
    \deg F^*\mathcal{O}_{\cp^{N-1}}(1)=\gamma+N-1,
\end{align*}
whose image is contained in no projective hyperplane. Composing $F$
with the rank-one projection map
\begin{align*}
    P([z]):=\frac{zz^*}{z^*z}
\end{align*}
gives a map into the real vector space
$\operatorname{Herm}(N)$ of Hermitian matrices. The integral
$\int_\Sigma|d(P\circ F)|^2\,\omega_\Sigma$ equals
$4\pi(\gamma+N-1)$ and is unchanged by projective linear
transformations of $F$.

Given a nonnegative integrable weight $W\ge V_+$ with
$m:=\int_\Sigma W\,\omega_\Sigma>0$, projective balancing gives
$G\in\operatorname{GL}(N,\c)$ such that
$\Psi:=P\circ G\circ F$ satisfies
\begin{align*}
    \int_\Sigma W\Psi\,\omega_\Sigma
    =\frac{m}{N}\operatorname{Id}_N.
\end{align*}
The balancing identity controls every matrix direction:
\begin{align*}
    \int_\Sigma W\bigl(\operatorname{tr}(A\Psi)\bigr)^2
    \,\omega_\Sigma
    \le\frac{m}{N}\operatorname{tr}(A^2),
    \qquad A\in\operatorname{Herm}(N).
\end{align*}
Projecting the functions $\operatorname{tr}(A\Psi)$ onto the negative
spectral subspace of $-\Delta_h-V$ defines a linear map of rank at
most $I_V$. We choose a Hilbert--Schmidt orthonormal basis adapted
to its kernel. All but at most $I_V$ of the corresponding coordinate
functions have nonnegative quadratic form, while each remaining
coordinate contributes at worst $-m/N$. Summing over the full
matrix basis yields
\begin{align*}
    \int_\Sigma V\,\omega_\Sigma
    \le4\pi(\gamma+N-1)+\frac{I_V}{N}m.
\end{align*}
Taking $W=V_++\varepsilon$ and letting $\varepsilon\downarrow0$
proves Theorem~\ref{Theorem: Scalar spectral inequality}.

To pass to minimal surfaces, the Gauss equation expresses the scalar
Jacobi potential as
\begin{align*}
    V_\Sigma
    &=|A_\Sigma|^2+\operatorname{Ric}_M(\nu,\nu)
      =-2K_\Sigma+b_\Sigma,\\
    b_\Sigma
    &=\operatorname{Ric}_M(e_1,e_1)
      +\operatorname{Ric}_M(e_2,e_2),
\end{align*}
where $e_1,e_2$ is a local orthonormal tangent frame and $\nu$ is a
local unit normal. In the orientable two-sided case,
$I_{V_\Sigma}=\operatorname{Ind}(\Sigma)$. If this index is positive,
the choice $N=4I_{V_\Sigma}$ in
Theorem~\ref{Theorem: Scalar spectral inequality}, together with
Gauss--Bonnet, gives
\begin{align*}
    8\pi\bigl(\gamma(\Sigma)-1\bigr)
    \le64\pi\operatorname{Ind}(\Sigma)
    +\int_\Sigma(V_\Sigma)_-\,\omega_\Sigma
    -3\int_\Sigma b_\Sigma\,\omega_\Sigma.
\end{align*}
The curvature assumptions control the last two terms. The stable
case follows by testing the second variation with the constant function one.
This proves Theorem~\ref{Theorem: Explicit two sided estimate} and
the Ricci curvature estimates in \eqref{Equation: Main Ricci genus estimates}.

For positive scalar curvature, a different form of the Gauss equation
is useful. Writing $R_M$ for the ambient scalar curvature, set
\begin{align*}
    W_\Sigma:=\frac12\bigl(R_M+|A_\Sigma|^2\bigr),
    \qquad S:=\int_\Sigma W_\Sigma\,\omega_\Sigma.
\end{align*}
Then $V_\Sigma=W_\Sigma-K_\Sigma$, and compactness and positivity of
$R_M$ give $(V_\Sigma)_+\le cW_\Sigma$ for a constant $c$ depending
only on the ambient metric. Substituting in
Theorem~\ref{Theorem: Scalar spectral inequality}, Gauss--Bonnet
cancels the genus terms and gives
\begin{align*}
    S\le4\pi N+\frac{c\operatorname{Ind}(\Sigma)}{N}S.
\end{align*}
For positive index, choosing $N$ proportional to the index bounds
$S$ linearly by the index. This controls both area and total
curvature, and the Gauss equation then bounds the genus. The stable
case again follows by testing with the constant function one. This
proves Theorem~\ref{Theorem: Positive scalar curvature}.

For the general embedded case, we use local unit normals on a fixed
finite cover of the ambient manifold by coordinate balls. A
partition of unity is used to compare the usual Morse index with the index
of a scalar operator on an orientable cover of the surface, after
subtracting a fixed constant from the scalar potential. The
localization error is bounded solely in terms of the ambient
metric. The shifted scalar estimate therefore introduces only an
additional area term. This yields Theorem~\ref{Theorem: Topology index and area}, including its nonorientable and one-sided cases.

\subsection{Related questions}
The preceding results leave several natural questions.
First, the analogous objective in higher dimensions would be to control all Betti numbers of a minimal hypersurface by its index and volume. More precisely, if $(M^{n+1},\bar g)$ is closed and $3\le n\le6$, does there exist $C=C(M,\bar g)>0$ such that every closed smoothly embedded minimal hypersurface $\Sigma\subset M$ satisfies
\begin{align*}
    \sum_{j=0}^n b_j(\Sigma;\mathbb{Z}_2)
    \le C\bigl(\operatorname{Ind}(\Sigma)
    +\operatorname{Vol}_n(\Sigma)\bigr)?
\end{align*}
In dimension two, Theorem~\ref{Theorem: Topology index and area}
also controls the total Betti number: each connected component has
a uniform positive lower area bound, and its zeroth and second
mod-$2$ Betti numbers both equal one. In higher dimensions,
the volume dependence in Song's estimate \cite{Song} does not
give the displayed additive inequality. Our proof uses both
the complex structure of an oriented Riemannian surface and the Gauss--Bonnet formula, so it does not directly extend to this setting.

The distinction between an ambient-dependent constant and a
dimension-dependent constant becomes essential in higher
dimensions. M\'aximo--Reiser--Semola constructed, in every ambient
dimension at least four, sequences of closed positively
Ricci-curved manifolds containing two-sided index-one minimal
hypersurfaces whose first Betti numbers tend to infinity
\cite{MRS26}. Thus the universal constant in
\eqref{Equation: Main Ricci genus estimates} has no direct
higher-dimensional counterpart. Their examples vary the ambient
manifold and do not settle the fixed-ambient
Marques--Neves--Schoen conjecture.

Second, the optimal universal coefficient in the positive-Ricci
genus estimate remains to be determined. Let $C_*$ be the infimum
of the constants $C$ such that
$\gamma(\Sigma)\le C\operatorname{Ind}(\Sigma)$ for every closed
connected orientable two-sided minimal immersion into every
closed Riemannian three-manifold of positive Ricci curvature.
Our estimate and Ros's genus-three index-one example \cite{Ros25}
give
\begin{align*}
    3\le C_*\le8.
\end{align*}
It is natural to ask whether the coefficient $8$ can be improved
and to determine the value of $C_*$.

The paper is organized as follows. In
Section~\ref{Section: Projective test functions}, we construct the
holomorphic test maps, prove the balancing and index estimates, and
establish Theorem~\ref{Theorem: Scalar spectral inequality}.
Section~\ref{Section: Curvature potentials} treats curvature
potentials and proves the estimates for orientable two-sided minimal
immersions, including the area and total curvature bounds under
positive scalar curvature. Finally, Section~\ref{Section: General embedded case}
contains the localization argument and the proof of
Theorem~\ref{Theorem: Topology index and area}. 
\renewcommand\labelitemi{$\vcenter{\hbox{\tiny$\bullet$}}$}
\begingroup
\renewcommand{\addcontentsline}[3]{}
\subsection*{On the use of AI}
The author used ChatGPT-6 Astra solely to improve the clarity and style of the prose, to assist in identifying potential errors or gaps in the mathematical arguments. The research ideas and all the proofs were developed without the use of AI tools. All mathematical arguments were independently checked by the author, who takes full responsibility for the content of this paper.
\par
\endgroup
\renewcommand\labelitemi{$\vcenter{\hbox{\tiny$\bullet$}}$}
\begingroup
\renewcommand{\addcontentsline}[3]{}
\subsection*{Acknowledgements}
I am grateful to Felix Schulze for his support and advice during the final stages of preparing this paper. I also thank Alessandro Pigati for his careful reading of the manuscript and his valuable suggestions for improvement. Finally, I would like to thank Benjamin Sharp for our insightful conversation on the positive scalar curvature case.
\par
\endgroup
\section{Projective test functions and the scalar spectral estimate}
\label{Section: Projective test functions}

Throughout this section, $(\Sigma,h)$ is a closed connected oriented
Riemannian surface of genus $\gamma$, endowed with the complex structure
determined by its metric and orientation. We denote its area form by
$\omega_\Sigma$. For $V\in C^\infty(\Sigma,\mathbb{R})$, we set
\begin{align*}
    Q_V(f):=\int_\Sigma\bigl(|\nabla f|_h^2-Vf^2\bigr)\,
    \omega_\Sigma,
    \qquad
    I_V:=\operatorname{Ind}(Q_V),
\end{align*}
where the index is computed on real-valued functions. Equivalently,
$I_V$ is the number of strictly negative eigenvalues, counted with
multiplicity, of $\mathcal{L}_V:=-\Delta_h-V$, with
$\Delta_h=\operatorname{div}_h\nabla$. We also write
$V_+:=\max\{V,0\}$ and $V_-:=\max\{-V,0\}$.

We first construct holomorphic maps into projective spaces of arbitrary
dimension, with degree controlled by the genus. We then balance the
associated projection coordinates and use the negative spectral subspace
of $\mathcal{L}_V$ to estimate their contribution to $Q_V$.

Recall that a point $p\in\Sigma$ is a \emph{base point} of a
holomorphic line bundle $L\to\Sigma$ if $s(p)=0$ for every
$s\in H^0(\Sigma,L)$. The bundle $L$ is called
\emph{globally generated} if, for every $p\in\Sigma$, the evaluation
map
\begin{align*}
    \operatorname{ev}_p:H^0(\Sigma,L)&\rightarrow L_p\\
    s&\mapsto s(p)
\end{align*}
is surjective. Since each fiber $L_p$ is one-dimensional, this is
equivalent to the absence of base points: at every point, at least
one global holomorphic section must be nonzero. Finally, for $p\in\Sigma$, we write
$L(-p):=L\otimes\mathcal{O}_\Sigma(-p)$; its holomorphic sections
identify naturally with the holomorphic sections of $L$ vanishing
at $p$, and $\deg L(-p)=\deg L-1$.

\begin{lemma}[Holomorphic maps of controlled degree]
\label{Lemma: Holomorphic maps of controlled degree}
    For every integer $N\ge2$, there exists a globally generated
    holomorphic line bundle $L\to\Sigma$ such that
    \begin{align}\label{Equation: Controlled degree line bundle}
        \deg L=\gamma+N-1,
        \qquad H^1(\Sigma,L)=0,
        \qquad h^0(\Sigma,L)=N.
    \end{align}
    In particular, there exists a holomorphic map
    $F:\Sigma\to\cp^{N-1}$ whose image is contained in no projective
    hyperplane and for which
    $F^*\mathcal{O}_{\cp^{N-1}}(1)\cong L$.
\end{lemma}
\begin{proof}
    Set $d:=\gamma+N-1$. If $\gamma=0$, identify $\Sigma$ with
    $\cp^1$ and take $L=\mathcal{O}_{\cp^1}(N-1)$. We may therefore
    assume that $\gamma\ge1$.

    Let $\mathcal{K}_\Sigma$ denote the canonical bundle of $\Sigma$.
    We will exclude two proper analytic subsets of the Picard variety
    $\operatorname{Pic}^d(\Sigma)$, which has complex dimension $\gamma$.
    Recall that a holomorphic line bundle $L\to\Sigma$ is called \emph{special} if $H^1(\Sigma,L)\neq0$, and \emph{nonspecial} otherwise. By Serre duality, $L$ is special precisely when $\mathcal{K}_\Sigma\otimes L^{-1}$ admits a nonzero holomorphic section. Since
    \begin{align*}
        \deg\bigl(\mathcal{K}_\Sigma\otimes L^{-1}\bigr)
        =2\gamma-2-d=\gamma-N-1,
    \end{align*}
    the special locus is empty if $\gamma-N-1<0$. Otherwise, it is
    the image of the holomorphic map
    \begin{align*}
        \operatorname{Sym}^{\gamma-N-1}(\Sigma)
        &\rightarrow\operatorname{Pic}^d(\Sigma)\\
        D&\mapsto
        \mathcal{K}_\Sigma\otimes\mathcal{O}_\Sigma(-D).
    \end{align*}
    Here $\operatorname{Sym}^k(\Sigma)$ parametrizes effective divisors
    of degree $k$. Since the source is compact, its image is an analytic
    subset of dimension at most $\gamma-N-1<\gamma$.

    For any line bundle outside this locus, Riemann--Roch gives
    $h^0(\Sigma,L)=d+1-\gamma=N$. Suppose that such an $L$ has a base
    point $p\in\Sigma$. Then
    \begin{align*}
        h^0\bigl(\Sigma,L(-p)\bigr)=h^0(\Sigma,L)=N.
    \end{align*}
    Applying Riemann--Roch to $L(-p)$, we obtain
    \begin{align*}
        h^1\bigl(\Sigma,L(-p)\bigr)
        =N-(d-\gamma)=1.
    \end{align*}
    Serre duality therefore provides a nonzero holomorphic section of
    $\mathcal{K}_\Sigma\otimes L^{-1}(p)$. Its zero divisor $D$ has
    degree $\gamma-N$, and
    \begin{align*}
        L\cong\mathcal{K}_\Sigma\otimes\mathcal{O}_\Sigma(p-D).
    \end{align*}
    In particular, no such base point exists when $\gamma-N<0$.
    Otherwise, the nonspecial line bundles with a base point are
    contained in the image of
    \begin{align*}
        \Sigma\times\operatorname{Sym}^{\gamma-N}(\Sigma)
        &\rightarrow\operatorname{Pic}^d(\Sigma)\\
        (p,D)&\mapsto
        \mathcal{K}_\Sigma\otimes\mathcal{O}_\Sigma(p-D).
    \end{align*}
    This image is analytic and has dimension at most
    $1+\gamma-N<\gamma$. We can thus choose $L$ outside the union
    of these two proper analytic subsets. Such an $L$ is nonspecial
    and has no base points, and hence is globally generated.
    This proves \eqref{Equation: Controlled degree line bundle}.

    Choose a basis $s_1,\ldots,s_N$ of $H^0(\Sigma,L)$. Since these
    sections have no common zero, they define a holomorphic map
    \begin{align*}
        F:\Sigma&\rightarrow\cp^{N-1}, \quad F(p)=[s_1(p):\cdots:s_N(p)],
    \end{align*}
    where the homogeneous coordinates are computed in any local
    holomorphic frame of $L$. The evaluation map identifies
    $F^*\mathcal{O}_{\cp^{N-1}}(1)$ with $L$. Finally, a hyperplane
    containing $F(\Sigma)$ would give a nontrivial linear relation
    among $s_1,\ldots,s_N$, contradicting their linear independence.
\end{proof}

We call a holomorphic map to projective space \emph{linearly full} if
its image is contained in no projective hyperplane. For $N\ge2$, let
$\operatorname{Herm}(N)$ denote the real vector space of Hermitian
$N\times N$ matrices, equipped with the Hilbert--Schmidt inner product
\begin{align*}
    \langle A,B\rangle_{\operatorname{Herm}(N)}
    :=\operatorname{tr}(AB).
\end{align*}
We consider the projection map
\begin{align}\label{Equation: Projective projection map}
    P:\cp^{N-1}\rightarrow\operatorname{Herm}(N),
    \quad P([z]):=\frac{zz^*}{z^*z},
\end{align}
where $z\in\c^N\setminus\{0\}$ is viewed as a column vector and
$z^*$ is its conjugate transpose.

The following balancing statement is the projective-space case of the Bourguignon--Li--Yau averaging theorem of Biliotti and Ghigi \cite[Section~3, Definition 70 and Theorem 89]{BG13}. We give a direct proof under the measure
hypothesis needed here.

\begin{lemma}[Projective balancing]
\label{Lemma: Projective balancing}
    Let $\mu$ be a Borel probability measure on $\cp^{N-1}$ assigning
    zero mass to every projective hyperplane. Then there exists
    $G\in\operatorname{GL}(N,\c)$ such that
    \begin{align}\label{Equation: Projective balancing identity}
        \int_{\cp^{N-1}}P([Gz])\,d\mu([z])
        =\frac1N\operatorname{Id}_N.
    \end{align}
\end{lemma}
\begin{proof}
    On the space
    \begin{align*}
        \mathcal{P}:=
        \{B\in\operatorname{Herm}(N):B>0,\ \det B=1\},
    \end{align*}
    define
    \begin{align*}
        \Phi(B):=\int_{\cp^{N-1}}
        \log\frac{z^*Bz}{z^*z}\,d\mu([z]).
    \end{align*}
    The integrand is independent of the representative $z$, and is
    bounded between the logarithms of the smallest and largest
    eigenvalues of $B$. In particular, $\Phi$ is finite and continuous.

    We first prove that its sublevel sets are compact. Suppose, to
    the contrary, that there is a sequence $(B_j)\subset\mathcal{P}$
    leaving every compact subset of $\mathcal{P}$ while $\Phi(B_j)$
    remains bounded above. Write
    \begin{align*}
        B_j&=U_j\operatorname{diag}
        (e^{\lambda_{1,j}},\ldots,e^{\lambda_{N,j}})U_j^*,
        \qquad U_j\in\operatorname{U}(N),\\
        \lambda_{1,j}&\ge\cdots\ge\lambda_{N,j},
        \qquad\sum_{a=1}^N\lambda_{a,j}=0.
    \end{align*}
    Then $r_j:=\lambda_{1,j}-\lambda_{N,j}\to+\infty$.
    Indeed, boundedness of $r_j$, together with the zero-sum condition,
    would bound every eigenvalue of $B_j$ from above and away from zero.
    After passing to a subsequence, we may assume that
    \begin{align*}
        U_j\rightarrow U,
        \qquad
        \frac{\lambda_{a,j}}{r_j}\rightarrow a_a
        \quad\text{for }a=1,\ldots,N.
    \end{align*}
    The limiting numbers satisfy
    \begin{align*}
        a_1\ge\cdots\ge a_N,
        \qquad\sum_{a=1}^N a_a=0,
        \qquad a_1-a_N=1.
    \end{align*}
    Thus $0\le(N-1)a_1+a_N=Na_1-1$, and $a_1\ge1/N$.

    Let
    \begin{align*}
        H:=\{[z]\in\cp^{N-1}:(U^*z)_1=0\}.
    \end{align*}
    For every $[z]\notin H$, the first coordinate of $U_j^*z$
    converges to a nonzero number. Therefore, for all sufficiently
    large $j$,
    \begin{align*}
        \lambda_{1,j}
        +\log\frac{|(U_j^*z)_1|^2}{z^*z}
        \le\log\frac{z^*B_jz}{z^*z}
        \le\lambda_{1,j}.
    \end{align*}
    Dividing by $r_j$ gives
    \begin{align*}
        \frac1{r_j}\log\frac{z^*B_jz}{z^*z}
        \rightarrow a_1
        \qquad\text{for every }[z]\notin H.
    \end{align*}
    Moreover, the zero-sum condition implies
    $\lambda_{N,j}\le0\le\lambda_{1,j}$, so
    \begin{align*}
        -1\le\frac{\lambda_{N,j}}{r_j}
        \le\frac1{r_j}\log\frac{z^*B_jz}{z^*z}
        \le\frac{\lambda_{1,j}}{r_j}\le1.
    \end{align*}
    Since $\mu(H)=0$, the dominated convergence theorem yields
    \begin{align*}
        \frac{\Phi(B_j)}{r_j}\rightarrow a_1\ge\frac1N.
    \end{align*}
    This contradicts the assumed upper bound for $\Phi(B_j)$.
    Hence the sublevel sets of $\Phi$ are compact, and $\Phi$ attains
    its minimum at some $B\in\mathcal{P}$.

    For any trace-free $A\in\operatorname{Herm}(N)$, the curve
    $B_t:=B^{1/2}e^{tA}B^{1/2}$ lies in $\mathcal{P}$. Differentiating
    at $t=0$, we obtain
    \begin{align*}
        0=\left.\frac{d}{dt}\right|_{t=0}\Phi(B_t)=\int_{\cp^{N-1}}
        \frac{z^*B^{1/2}AB^{1/2}z}{z^*Bz}\,d\mu([z])=\operatorname{tr}\left(
        A\int_{\cp^{N-1}}P([B^{1/2}z])\,d\mu([z])\right).
    \end{align*}
    Differentiation under the integral is justified by the uniform
    bounds for the integrand and its derivative for $t$ in a compact
    neighborhood of zero. The averaged matrix is therefore orthogonal
    to all trace-free Hermitian matrices and hence is a scalar multiple
    of $\operatorname{Id}_N$. Its trace is one. Taking $G:=B^{1/2}$
    proves \eqref{Equation: Projective balancing identity}.
\end{proof}

We normalize the Fubini--Study form on $\cp^{N-1}$ by
\begin{align}\label{Equation: Fubini Study normalization}
    \omega_{\mathrm{FS}}
    :=\frac{i}{2}\partial\bar\partial\log(1+|w|^2),
    \qquad\int_{\cp^1}\omega_{\mathrm{FS}}=\pi,
\end{align}
and denote its associated Riemannian metric by $g_{\mathrm{FS}}$.
The next lemma records the identities used in the spectral argument.
For the corresponding moment-map energy formula, see also
\cite[Equation (18)]{AGL}.

\begin{lemma}[Projection coordinates and energy]
\label{Lemma: Projection coordinates and energy}
    Let $\Psi:\Sigma\to\operatorname{Herm}(N)$ be a smooth map taking
    values in rank-one orthogonal projections. For an orthonormal basis
    $A_1,\ldots,A_{N^2}$ of $\operatorname{Herm}(N)$, set
    $f_\alpha:=\operatorname{tr}(A_\alpha\Psi)$. Then
    \begin{align}\label{Equation: Projection Parseval identities}
        \sum_{\alpha=1}^{N^2}f_\alpha^2=1,
        \qquad
        \sum_{\alpha=1}^{N^2}|\nabla f_\alpha|_h^2=|d\Psi|_h^2.
    \end{align}
    Moreover, for every $A\in\operatorname{Herm}(N)$,
    \begin{align}\label{Equation: Projection coordinate bound}
        |\operatorname{tr}(A\Psi)|^2
        \le\operatorname{tr}(A^2\Psi).
    \end{align}
    If $F:\Sigma\to\cp^{N-1}$ is holomorphic and
    $d:=\deg F^*\mathcal{O}_{\cp^{N-1}}(1)$, then
    \begin{align}\label{Equation: Holomorphic projection energy}
        \int_\Sigma|d(P\circ F)|_h^2\,\omega_\Sigma=4\pi d.
    \end{align}
    The same energy identity holds after any projective linear
    transformation of $F$.
\end{lemma}
\begin{proof}
    Since $\Psi^2=\Psi$ and $\operatorname{tr}(\Psi)=1$, we have
    $|\Psi|^2=\operatorname{tr}(\Psi^2)=1$. Parseval's identity gives
    the first equality in
    \eqref{Equation: Projection Parseval identities}. Applying
    Parseval's identity to $d\Psi(e_1)$ and $d\Psi(e_2)$, where
    $e_1,e_2$ is a local orthonormal tangent frame, gives the second.

    At any point, write $\Psi=zz^*$ with $z\in\c^N$ and $z^*z=1$.
    Since $A=A^*$, the Cauchy--Schwarz inequality yields
    \begin{align*}
        |\operatorname{tr}(A\Psi)|^2
        =|z^*Az|^2
        \le|Az|^2
        =z^*A^2z
        =\operatorname{tr}(A^2\Psi),
    \end{align*}
    proving \eqref{Equation: Projection coordinate bound}.

    To compute the energy, use the affine chart $[1:w]$ centered at
    $[1:0:\cdots:0]$. In this chart,
    \begin{align*}
        P([1:w])=\frac1{1+|w|^2}
        \begin{pmatrix}
            1&w^*\\
            w&ww^*
        \end{pmatrix}.
    \end{align*}
    Thus, for $\xi\in\c^{N-1}$,
    \begin{align*}
        dP_0(\xi)=
        \begin{pmatrix}
            0&\xi^*\\
            \xi&0
        \end{pmatrix},
        \qquad |dP_0(\xi)|^2=2|\xi|^2.
    \end{align*}
    Under the normalization
    \eqref{Equation: Fubini Study normalization},
    $g_{\mathrm{FS},0}(\xi,\xi)=|\xi|^2$. Both metrics are invariant
    under the unitary action, which is transitive on $\cp^{N-1}$.
    Consequently,
    \begin{align*}
        P^*g_{\operatorname{Herm}(N)}=2g_{\mathrm{FS}}.
    \end{align*}
    Since $F$ is holomorphic, its differential is complex linear and
    satisfies the pointwise identity
    \begin{align*}
        |dF|_{h,g_{\mathrm{FS}}}^2\,\omega_\Sigma
        =2F^*\omega_{\mathrm{FS}}.
    \end{align*}
    The identity also holds wherever $dF=0$, since both sides vanish there.
    As $[\omega_{\mathrm{FS}}/\pi]
    =c_1(\mathcal{O}_{\cp^{N-1}}(1))$, we obtain
    \begin{align*}
        \int_\Sigma|d(P\circ F)|_h^2\,\omega_\Sigma=2\int_\Sigma|dF|_{h,g_{\mathrm{FS}}}^2\,\omega_\Sigma=4\int_\Sigma F^*\omega_{\mathrm{FS}}
        =4\pi d.
    \end{align*}
    Finally, a projective linear transformation $T_G$ is holomorphic
    and satisfies
    $T_G^*\mathcal{O}_{\cp^{N-1}}(1)
    \cong\mathcal{O}_{\cp^{N-1}}(1)$. Hence $T_G\circ F$ is
    holomorphic of the same degree, and the preceding computation
    applies to it as well.
\end{proof}

\begin{proposition}[Index control for balanced projection coordinates]
\label{Proposition: Balanced projection index estimate}
    Let $V\in C^\infty(\Sigma,\mathbb{R})$, and let
    $\Psi:\Sigma\to\operatorname{Herm}(N)$ be smooth and take values
    in rank-one orthogonal projections. Suppose that
    $W\in L^1(\Sigma,\omega_\Sigma)$ satisfies $W\ge\max\{V,0\}$
    almost everywhere, $\int_\Sigma W\,\omega_\Sigma>0$, and
    \begin{align}\label{Equation: Balanced weight hypothesis}
        \int_\Sigma W\Psi\,\omega_\Sigma
        =\frac1N\left(\int_\Sigma W\,\omega_\Sigma\right)
        \operatorname{Id}_N.
    \end{align}
    Then
    \begin{align}\label{Equation: Balanced projection index estimate}
        \int_\Sigma V\,\omega_\Sigma
        \le\int_\Sigma|d\Psi|_h^2\,\omega_\Sigma
        +\frac{I_V}{N}\int_\Sigma W\,\omega_\Sigma.
    \end{align}
\end{proposition}
\begin{proof}
    Write $m_W:=\int_\Sigma W\,\omega_\Sigma$. For
    $A\in\operatorname{Herm}(N)$, set $f_A:=\operatorname{tr}(A\Psi)$.
    By \eqref{Equation: Projection coordinate bound}, the
    nonnegativity of $W$, and
    \eqref{Equation: Balanced weight hypothesis},
    \begin{align}\label{Equation: Weighted projection coordinate estimate}
        \int_\Sigma Wf_A^2\,\omega_\Sigma\le\int_\Sigma W\operatorname{tr}(A^2\Psi)\,
        \omega_\Sigma=\operatorname{tr}\left(
        A^2\int_\Sigma W\Psi\,\omega_\Sigma\right)
        =\frac{m_W}{N}\operatorname{tr}(A^2).
    \end{align}

    Let $E_-$ be the real span of the eigenfunctions of
    $\mathcal{L}_V$ associated with strictly negative eigenvalues,
    and let $\Pi_-$ be the orthogonal projection onto $E_-$ in
    $L^2(\Sigma,\omega_\Sigma)$. The spectral theorem gives
    $\dim E_-=I_V$ and
    \begin{align}\label{Equation: Nonnegative spectral complement}
        Q_V(f)\ge0
        \qquad\text{whenever }f\in C^\infty(\Sigma,\mathbb{R})
        \text{ and }\Pi_-f=0.
    \end{align}
    Consider the real linear map
    \begin{align*}
        \mathcal{T}:\operatorname{Herm}(N)\rightarrow E_-,
        \qquad\mathcal{T}(A):=\Pi_-f_A,
    \end{align*}
    and set $r:=\operatorname{rank}\mathcal{T}\le I_V$. Choose an
    orthonormal basis $A_1,\ldots,A_{N^2}$ of
    $\operatorname{Herm}(N)$ such that
    $A_{r+1},\ldots,A_{N^2}$ span $\ker\mathcal{T}$, and write
    $f_\alpha:=f_{A_\alpha}$. For $\alpha>r$,
    \eqref{Equation: Nonnegative spectral complement} gives
    \begin{align*}
        \int_\Sigma Vf_\alpha^2\,\omega_\Sigma
        \le\int_\Sigma|\nabla f_\alpha|_h^2\,\omega_\Sigma.
    \end{align*}
    For $\alpha\le r$, the inequality $V\le W$ and
    \eqref{Equation: Weighted projection coordinate estimate} imply
    \begin{align*}
        \int_\Sigma Vf_\alpha^2\,\omega_\Sigma
        \le\int_\Sigma Wf_\alpha^2\,\omega_\Sigma
        \le\frac{m_W}{N},
    \end{align*}
    since $\operatorname{tr}(A_\alpha^2)=1$. Summing these inequalities
    and using \eqref{Equation: Projection Parseval identities}, we
    conclude that
    \begin{align*}
        \int_\Sigma V\,\omega_\Sigma=\sum_{\alpha=1}^{N^2}
        \int_\Sigma Vf_\alpha^2\,\omega_\Sigma\le\sum_{\alpha=r+1}^{N^2}
        \int_\Sigma|\nabla f_\alpha|_h^2\,\omega_\Sigma
        +\frac{r\,m_W}{N}\le\int_\Sigma|d\Psi|_h^2\,\omega_\Sigma
        +\frac{I_V\,m_W}{N}.
    \end{align*}
    The orthonormal basis used here is a basis of the full real matrix
    space $\operatorname{Herm}(N)$. No $L^2$-orthogonality among the
    functions $f_\alpha$ is required.
\end{proof}

\begin{proposition}[The spectral estimate from a holomorphic map]
\label{Proposition: Holomorphic spectral estimate}
    Suppose that $\Sigma$ admits a linearly full holomorphic map
    $F:\Sigma\to\cp^{N-1}$, with $N\ge2$, and set
    $d:=\deg F^*\mathcal{O}_{\cp^{N-1}}(1)$. Then, for every
    $V\in C^\infty(\Sigma,\mathbb{R})$,
    \begin{align}\label{Equation: Holomorphic spectral estimate}
        \int_\Sigma V\,\omega_\Sigma
        \le4\pi d+\frac{I_V}{N}
        \int_\Sigma V_+\,\omega_\Sigma.
    \end{align}
    Equivalently,
    \begin{align}\label{Equation: Rearranged holomorphic estimate}
        (N-I_V)\int_\Sigma V\,\omega_\Sigma
        \le4\pi Nd+I_V\int_\Sigma V_-\,\omega_\Sigma.
    \end{align}
\end{proposition}
\begin{proof}
    Fix $\varepsilon>0$ and define
    \begin{align*}
        W_\varepsilon:=V_++\varepsilon,
        \qquad
        m_\varepsilon:=\int_\Sigma W_\varepsilon\,\omega_\Sigma,
        \qquad
        \mu_\varepsilon:=F_\#\left(
        \frac{W_\varepsilon}{m_\varepsilon}\,\omega_\Sigma\right).
    \end{align*}
    Thus $\mu_\varepsilon$ is a Borel probability measure on
    $\cp^{N-1}$. If $H\subset\cp^{N-1}$ is a hyperplane, its defining
    linear form pulls back to a holomorphic section of
    $F^*\mathcal{O}_{\cp^{N-1}}(1)$. This section is not identically
    zero because $F$ is linearly full. Hence $F^{-1}(H)$ is finite,
    and $\mu_\varepsilon(H)=0$.

    By Lemma~\ref{Lemma: Projective balancing}, there exists
    $G_\varepsilon\in\operatorname{GL}(N,\c)$ such that, writing
    $T_{G_\varepsilon}([z]):=[G_\varepsilon z]$ and
    $\Psi_\varepsilon:=P\circ T_{G_\varepsilon}\circ F$, we have
    \begin{align*}
        \int_\Sigma W_\varepsilon\Psi_\varepsilon\,\omega_\Sigma
        =\frac{m_\varepsilon}{N}\operatorname{Id}_N.
    \end{align*}
    Applying
    Proposition~\ref{Proposition: Balanced projection index estimate}
    and Lemma~\ref{Lemma: Projection coordinates and energy} gives
    \begin{align*}
        \int_\Sigma V\,\omega_\Sigma\le\int_\Sigma|d\Psi_\varepsilon|_h^2\,\omega_\Sigma
        +\frac{I_V}{N}m_\varepsilon=4\pi d+\frac{I_V}{N}
        \left(\int_\Sigma V_+\,\omega_\Sigma
        +\varepsilon\operatorname{Area}_h(\Sigma)\right).
    \end{align*}
    Letting $\varepsilon\downarrow0$ proves
    \eqref{Equation: Holomorphic spectral estimate}. This passage
    requires no convergence of $G_\varepsilon$, since the final
    upper bound is independent of the chosen balancing transformation.
    Finally, \eqref{Equation: Rearranged holomorphic estimate} follows
    by multiplying by $N$ and using $V_+=V+V_-$.
\end{proof}

\begin{proof}[Proof of Theorem~\ref{Theorem: Scalar spectral inequality}]
    Fix an integer $N\ge2$. By
    Lemma~\ref{Lemma: Holomorphic maps of controlled degree}, there
    exists a linearly full holomorphic map
    $F:\Sigma\to\cp^{N-1}$ of degree $\gamma+N-1$. Applying
    Proposition~\ref{Proposition: Holomorphic spectral estimate}
    to this map yields
    \begin{align*}
        \int_\Sigma V\,\omega_\Sigma
        \le4\pi(\gamma+N-1)
        +\frac{I_V}{N}\int_\Sigma V_+\,\omega_\Sigma,
    \end{align*}
    as claimed.
\end{proof}
\begin{remark}[Sharpness in the scalar estimate]
\label{Remark: Sharp normalization of the scalar estimate}
    Consider the totally geodesic equatorial sphere
    $\mathbb{S}^2\subset\mathbb{S}^3$, in the round three-sphere of constant sectional curvature $1$. Its induced metric is the unit round
    metric, and its Jacobi potential is
    \begin{align*}
        V=|A_{\mathbb{S}^2}|^2
        +\operatorname{Ric}_{\mathbb{S}^3}(\nu,\nu)=2.
    \end{align*}
    The operator $-\Delta-2$ has one strictly negative eigenvalue,
    corresponding to the constant functions, and a
    three-dimensional kernel. Thus $I_V=1$, and
    Theorem~\ref{Theorem: Scalar spectral inequality} with $N=2$
    gives equality:
    \begin{align*}
        \int_{\mathbb{S}^2}V\,\omega_{\mathbb{S}^2}
        =8\pi
        =4\pi+\frac12\,8\pi.
    \end{align*}
    Consequently, the coefficient $4\pi$ multiplying $\gamma+N-1$
    cannot be replaced by a smaller universal constant while
    retaining the term
    $(I_V/N)\int_\Sigma V_+\,\omega_\Sigma$
    and the same range of parameters.
\end{remark}
\section{Curvature potentials and two-sided minimal immersions}
\label{Section: Curvature potentials}

We now apply the scalar spectral inequality obtained in Section~\ref{Section: Projective test functions} to potentials of the form $V=-2K_\Sigma+b$, where $K_\Sigma$ is the Gauss curvature of $\Sigma$ and $b\in C^{\infty}(\Sigma,\mathbb{R})$. The notation $Q_V$ and $I_V$ is as in Section~\ref{Section: Projective test functions}.

\begin{proposition}[Genus control for a curvature potential]
\label{Proposition: Genus control for curvature potentials}
    Let $(\Sigma,h)$ be a closed connected oriented Riemannian surface
    of genus $\gamma$, with Gauss curvature $K_\Sigma$. Suppose that
    \begin{align}\label{Equation: Curvature potential decomposition}
        V=-2K_\Sigma+b,
        \qquad b\in C^\infty(\Sigma,\mathbb{R}),
        \qquad I:=\operatorname{Ind}(Q_V).
    \end{align}
    If $I\ge1$, then
    \begin{align}\label{Equation: Exact curvature potential estimate}
        8\pi(\gamma-1)
        \le64\pi I+\int_\Sigma V_-\,\omega_\Sigma
        -3\int_\Sigma b\,\omega_\Sigma.
    \end{align}
    If $I=0$, then
    \begin{align}\label{Equation: Stable curvature potential estimate}
        8\pi(\gamma-1)\le-\int_\Sigma b\,\omega_\Sigma.
    \end{align}
    In particular, if $b\ge-\beta$ and $V_-\le\eta$ for constants
    $\beta,\eta\ge0$, then
    \begin{align}\label{Equation: Uniform curvature potential estimate}
        \gamma\le8I+1+
        \frac{3\beta+\eta}{8\pi}\operatorname{Area}_h(\Sigma).
    \end{align}
\end{proposition}
\begin{proof}
    Set
    \begin{align*}
        T:=\int_\Sigma V\,\omega_\Sigma,
        \qquad B:=\int_\Sigma b\,\omega_\Sigma,
        \qquad D:=\int_\Sigma V_-\,\omega_\Sigma.
    \end{align*}
    By Gauss--Bonnet,
    \begin{align}\label{Equation: Integrated curvature potential}
        T=-2\int_\Sigma K_\Sigma\,\omega_\Sigma+B
        =8\pi(\gamma-1)+B.
    \end{align}

    Suppose first that $I\ge1$. We may then apply
    Theorem~\ref{Theorem: Scalar spectral inequality} with $N=4I$.
    Since $V_+=V+V_-$, this gives
    \begin{align*}
        T\le4\pi(\gamma+4I-1)+\frac14(T+D),
    \end{align*}
    and hence
    \begin{align*}
        3T\le16\pi(\gamma+4I-1)+D.
    \end{align*}
    Substituting \eqref{Equation: Integrated curvature potential},
    we obtain
    \begin{align*}
        24\pi(\gamma-1)+3B
        \le16\pi(\gamma-1)+64\pi I+D.
    \end{align*}
    Rearranging proves
    \eqref{Equation: Exact curvature potential estimate}.

    If $I=0$, the quadratic form $Q_V$ is nonnegative. Testing it with the constant function $1$ gives
    \begin{align*}
        0\le Q_V(1)=-\int_\Sigma V\,\omega_\Sigma=-T.
    \end{align*}
    Equation~\eqref{Equation: Integrated curvature potential} then
    yields \eqref{Equation: Stable curvature potential estimate}.

    Finally, assume that $b\ge-\beta$ and $V_-\le\eta$. If $I\ge1$,
    \eqref{Equation: Exact curvature potential estimate} implies
    \begin{align*}
        8\pi(\gamma-1)
        \le64\pi I+(3\beta+\eta)\operatorname{Area}_h(\Sigma),
    \end{align*}
    which is \eqref{Equation: Uniform curvature potential estimate}.
    If $I=0$, \eqref{Equation: Stable curvature potential estimate}
    gives the stronger bound
    \begin{align*}
        \gamma\le1+\frac{\beta}{8\pi}
        \operatorname{Area}_h(\Sigma).
    \end{align*}
    Since $\beta,\eta\ge0$, this also implies
    \eqref{Equation: Uniform curvature potential estimate}.
\end{proof}
\begin{remark}[The choice of projective dimension]
\label{Remark: Choice of projective dimension}
    The choice $N=4I$ in the preceding proof minimizes the
    contribution depending only on the index. Indeed, suppose
    that $I\ge1$, and write
    \begin{align*}
        B:=\int_\Sigma b\,\omega_\Sigma,
        \qquad
        D:=\int_\Sigma V_-\,\omega_\Sigma.
    \end{align*}
    For every integer $N>2I$, the scalar spectral inequality,
    together with
    $\int_\Sigma V\,\omega_\Sigma=8\pi(\gamma-1)+B$,
    gives
    \begin{align*}
        4\pi(N-2I)(\gamma-1)
        \le4\pi N^2+ID-(N-I)B.
    \end{align*}
    Dividing by $4\pi(N-2I)$, we obtain
    \begin{align*}
        \gamma-1
        \le\frac{N^2}{N-2I}
        +\frac{ID-(N-I)B}{4\pi(N-2I)}.
    \end{align*}
    The first term on the right is minimized at $N=4I$, since
    \begin{align*}
        \frac{N^2}{N-2I}-8I
        =\frac{(N-4I)^2}{N-2I}\ge0.
    \end{align*}
    This explains the coefficient $8$ in the resulting genus
    estimate. It does not establish the optimality of that
    coefficient in the geometric inequalities.
\end{remark}
\begin{lemma}[The Jacobi potential and ambient curvature]
\label{Lemma: Jacobi potential and ambient curvature}
    Let $\iota:\Sigma\to(M^3,\bar g)$ be a minimal immersion of a
    closed surface, and let $K_\Sigma$ and $A_\Sigma$ denote its
    intrinsic Gauss curvature and second fundamental form. For any
    locally defined unit normal $\nu$, set
    \begin{align}\label{Equation: Scalar Jacobi potential}
        V_\Sigma:=|A_\Sigma|^2+
        \operatorname{Ric}_M(\nu,\nu).
    \end{align}
    This defines a smooth function on $\Sigma$, independently of
    the local choices of $\nu$, and
    \begin{align}
        V_\Sigma&=-2K_\Sigma+b_\Sigma,
        \label{Equation: Jacobi curvature decomposition}\\
        b_\Sigma&:=2\operatorname{sec}_M(T\Sigma)
        +\operatorname{Ric}_M(\nu,\nu)=\operatorname{Ric}_M(e_1,e_1)
        +\operatorname{Ric}_M(e_2,e_2),
        \label{Equation: Ambient tangential Ricci trace}
    \end{align}
    where $e_1,e_2$ is any local orthonormal tangent frame.
    Writing $R_M$ for the ambient scalar curvature, one also has
    \begin{align}\label{Equation: Scalar curvature Jacobi decomposition}
        V_\Sigma=-K_\Sigma
        +\frac12\bigl(R_M+|A_\Sigma|^2\bigr).
    \end{align}
    If all ambient sectional curvatures along the immersion are at
    least $-\kappa$, for some $\kappa\ge0$, then, for every $c\ge0$,
    \begin{align}\label{Equation: Shifted Jacobi curvature bounds}
        b_\Sigma-c\ge-4\kappa-c,
        \qquad (V_\Sigma-c)_-\le2\kappa+c.
    \end{align}
    If the immersion is two-sided, then
    \begin{align}\label{Equation: Scalar and normal index agreement}
        I_{V_\Sigma}=\operatorname{Ind}(\Sigma).
    \end{align}
\end{lemma}
\begin{proof}
    Any two local unit normal fields differ by a sign on their common
    domain. Since $\operatorname{Ric}_M(\nu,\nu)$ is unchanged by
    replacing $\nu$ with $-\nu$, the local definitions of $V_\Sigma$
    agree on overlaps and give a smooth function on $\Sigma$.

    Let $\lambda_1,\lambda_2$ be the principal curvatures at a point.
    Minimality gives $\lambda_1+\lambda_2=0$, so the Gauss equation
    yields
    \begin{align*}
        K_\Sigma=\operatorname{sec}_M(T\Sigma)+\lambda_1\lambda_2=\operatorname{sec}_M(T\Sigma)
        -\frac12(\lambda_1^2+\lambda_2^2)
        =\operatorname{sec}_M(T\Sigma)-\frac12|A_\Sigma|^2.
    \end{align*}
    Rearranging and adding $\operatorname{Ric}_M(\nu,\nu)$ proves
    \eqref{Equation: Jacobi curvature decomposition} with
    $b_\Sigma=2\operatorname{sec}_M(T\Sigma)
    +\operatorname{Ric}_M(\nu,\nu)$.

    For an orthonormal frame $e_1,e_2,\nu$, the three-dimensional
    curvature identities give
    \begin{align*}
        \operatorname{Ric}_M(e_1,e_1)
        &=\operatorname{sec}_M(e_1\wedge e_2)
        +\operatorname{sec}_M(e_1\wedge\nu),\\
        \operatorname{Ric}_M(e_2,e_2)
        &=\operatorname{sec}_M(e_1\wedge e_2)
        +\operatorname{sec}_M(e_2\wedge\nu),\\
        \operatorname{Ric}_M(\nu,\nu)
        &=\operatorname{sec}_M(e_1\wedge\nu)
        +\operatorname{sec}_M(e_2\wedge\nu).
    \end{align*}
    Here $\operatorname{sec}_M(u\wedge v)$ denotes the sectional
    curvature of the plane spanned by the orthonormal vectors $u,v$.
    Adding the first two identities and using the third proves
    \eqref{Equation: Ambient tangential Ricci trace}.
    Moreover,
    $R_M=2\operatorname{sec}_M(T\Sigma)
    +2\operatorname{Ric}_M(\nu,\nu)$. Combining this identity with
    the Gauss equation gives
    \eqref{Equation: Scalar curvature Jacobi decomposition}.

    If $\operatorname{sec}_M\ge-\kappa$ along the immersion, these
    formulas imply $\operatorname{Ric}_M(\nu,\nu)\ge-2\kappa$.
    Consequently,
    \begin{align*}
        V_\Sigma\ge-2\kappa,
        \qquad b_\Sigma\ge-4\kappa.
    \end{align*}
    Thus $b_\Sigma-c\ge-4\kappa-c$ and
    $V_\Sigma-c\ge-2\kappa-c$, proving
    \eqref{Equation: Shifted Jacobi curvature bounds}.

    It remains to compare the indices in the two-sided case.
    Let $N\Sigma\subset\iota^*TM$ denote the normal line bundle,
    endowed with its induced connection $\nabla^\perp$. The second
    variation of area on a smooth normal section $X$ is
    \begin{align}\label{Equation: Normal Jacobi quadratic form}
        Q_\Sigma^\perp(X)=\int_\Sigma\bigl(
        |\nabla^\perp X|^2-|A_\Sigma|^2|X|^2
        -\operatorname{Ric}_M(X,X)\bigr)\,\omega_\Sigma=\int_\Sigma\bigl(
        |\nabla^\perp X|^2-V_\Sigma|X|^2\bigr)\,\omega_\Sigma;
    \end{align}
    see, for example, \cite[Section 2.1]{CKM17} for this expression.
    The second equality follows because the normal bundle has rank
    one, so $\operatorname{Ric}_M(X,X)
    =\operatorname{Ric}_M(\nu,\nu)|X|^2$ pointwise.

    When the immersion is two-sided, choose a global smooth unit
    normal $\nu$. Since $\nabla^\perp$ is a metric connection on a
    line bundle, $\nabla^\perp\nu=0$. Indeed, each covariant derivative
    of $\nu$ is normal and orthogonal to $\nu$, and hence vanishes.
    Every smooth normal section is uniquely of the form $X=f\nu$, and
    \begin{align*}
        \nabla^\perp(f\nu)=df\otimes\nu,
        \qquad Q_\Sigma^\perp(f\nu)=Q_{V_\Sigma}(f).
    \end{align*}
    Therefore, $f\mapsto f\nu$ is a linear isomorphism preserving
    the quadratic forms. Their negative indices agree, which proves
    \eqref{Equation: Scalar and normal index agreement}.
\end{proof}

\begin{corollary}[A shifted scalar Jacobi operator]
\label{Corollary: Shifted scalar Jacobi estimate}
    Let $\iota:\Sigma\to(M^3,\bar g)$ be a closed connected orientable
    minimal immersion. Suppose that all ambient sectional curvatures
    along the immersion are at least $-\kappa$, where $\kappa\ge0$.
    For $c\ge0$, let
    \begin{align*}
        J_c:=\operatorname{Ind}(Q_{V_\Sigma-c})
    \end{align*}
    be the scalar index computed on real-valued functions on $\Sigma$.
    Then
    \begin{align}\label{Equation: Shifted scalar Jacobi estimate}
        \gamma(\Sigma)
        \le8J_c+1+
        \frac{14\kappa+4c}{8\pi}\operatorname{Area}(\Sigma).
    \end{align}
    No two-sidedness assumption is required.
\end{corollary}
\begin{proof}
    Fix an orientation on $\Sigma$. By
    Lemma~\ref{Lemma: Jacobi potential and ambient curvature},
    \begin{align*}
        V_\Sigma-c=-2K_\Sigma+(b_\Sigma-c),
    \end{align*}
    with
    \begin{align*}
        b_\Sigma-c\ge-(4\kappa+c),
        \qquad (V_\Sigma-c)_-\le2\kappa+c.
    \end{align*}
    We may therefore apply
    Proposition~\ref{Proposition: Genus control for curvature potentials}
    to $V_\Sigma-c$, with $\beta=4\kappa+c$ and $\eta=2\kappa+c$.
    Equation~\eqref{Equation: Uniform curvature potential estimate}
    gives
    \begin{align*}
        \gamma(\Sigma)\le8J_c+1+
        \frac{3(4\kappa+c)+(2\kappa+c)}{8\pi}
        \operatorname{Area}(\Sigma)=8J_c+1+
        \frac{14\kappa+4c}{8\pi}\operatorname{Area}(\Sigma),
    \end{align*}
    as claimed.
\end{proof}

\begin{proof}[Proof of Theorem~\ref{Theorem: Explicit two sided estimate}]
    Since the immersion is two-sided,
    \eqref{Equation: Scalar and normal index agreement} gives
    $I_{V_\Sigma}=\operatorname{Ind}(\Sigma)$. Taking $c=0$ in
    Corollary~\ref{Corollary: Shifted scalar Jacobi estimate}, we obtain
    \begin{align*}
        \gamma(\Sigma)\le8\operatorname{Ind}(\Sigma)+1+
        \frac{14\kappa}{8\pi}\operatorname{Area}(\Sigma)=8\operatorname{Ind}(\Sigma)+1+
        \frac{7\kappa}{4\pi}\operatorname{Area}(\Sigma).
    \end{align*}

    If $\operatorname{Ind}(\Sigma)=0$, then $I_{V_\Sigma}=0$.
    Applying \eqref{Equation: Stable curvature potential estimate}
    and the bound $b_\Sigma\ge-4\kappa$, we find
    \begin{align*}
        8\pi\bigl(\gamma(\Sigma)-1\bigr)
        \le-\int_\Sigma b_\Sigma\,\omega_\Sigma
        \le4\kappa\operatorname{Area}(\Sigma).
    \end{align*}
    Dividing by $8\pi$ proves the stronger stable estimate
    \begin{align*}
        \gamma(\Sigma)
        \le1+\frac{\kappa}{2\pi}\operatorname{Area}(\Sigma).
    \end{align*}
\end{proof}

The expression of $b_\Sigma$ as the tangential trace of the ambient
Ricci tensor also gives an estimate involving only the index when
the ambient Ricci curvature is nonnegative.

\begin{corollary}[Nonnegative and positive ambient Ricci curvature]
\label{Corollary: Ricci curvature and genus}
    Let $\iota:\Sigma\to(M^3,\bar g)$ be a closed connected orientable
    two-sided minimal immersion. If $\operatorname{Ric}_M\ge0$ along
    the immersion, then
    \begin{align}\label{Equation: Nonnegative Ricci genus estimate}
        \gamma(\Sigma)\le8\operatorname{Ind}(\Sigma)+1.
    \end{align}
    If $\operatorname{Ric}_M>0$ along the immersion, then
    \begin{align}\label{Equation: Positive Ricci genus estimate}
        \gamma(\Sigma)\le8\operatorname{Ind}(\Sigma).
    \end{align}
\end{corollary}
\begin{proof}
    Write $I:=\operatorname{Ind}(\Sigma)=I_{V_\Sigma}$. By
    Lemma~\ref{Lemma: Jacobi potential and ambient curvature}, the
    assumption $\operatorname{Ric}_M\ge0$ implies
    \begin{align*}
        V_\Sigma\ge0,
        \qquad
        b_\Sigma=\operatorname{Ric}_M(e_1,e_1)
        +\operatorname{Ric}_M(e_2,e_2)\ge0.
    \end{align*}
    If $I\ge1$, \eqref{Equation: Exact curvature potential estimate}
    therefore gives
    \begin{align*}
        8\pi\bigl(\gamma(\Sigma)-1\bigr)
        \le64\pi I-3\int_\Sigma b_\Sigma\,\omega_\Sigma
        \le64\pi I.
    \end{align*}
    If $I=0$, \eqref{Equation: Stable curvature potential estimate}
    instead yields
    \begin{align*}
        8\pi\bigl(\gamma(\Sigma)-1\bigr)
        \le-\int_\Sigma b_\Sigma\,\omega_\Sigma\le0.
    \end{align*}
    These two cases prove
    \eqref{Equation: Nonnegative Ricci genus estimate}.

    Suppose now that $\operatorname{Ric}_M>0$ along the immersion.
    Then $V_\Sigma>0$ and $b_\Sigma>0$ everywhere. In particular,
    \begin{align*}
        Q_{V_\Sigma}(1)
        =-\int_\Sigma V_\Sigma\,\omega_\Sigma<0,
    \end{align*}
    so $I\ge1$. Applying
    \eqref{Equation: Exact curvature potential estimate} once more,
    we obtain the strict inequality
    \begin{align*}
        8\pi\bigl(\gamma(\Sigma)-1\bigr)
        \le64\pi I-3\int_\Sigma b_\Sigma\,\omega_\Sigma
        <64\pi I.
    \end{align*}
    Thus $\gamma(\Sigma)<8I+1$. Since $\gamma(\Sigma)$ and $I$ are
    integers, this implies
    \eqref{Equation: Positive Ricci genus estimate}.
\end{proof}
\begin{proof}[Proof of Theorem~\ref{Theorem: Positive scalar curvature}]
    Choose constants $r>0$ and $a\ge0$, depending only on the ambient
    metric, such that
    \begin{align*}
        R_M\ge r,
        \qquad \operatorname{Ric}_M\le a\bar g,
    \end{align*}
    and set
    \begin{align*}
        c:=2\max\{1,a/r\},
        \qquad m:=\lceil2c\rceil,
        \qquad I:=\operatorname{Ind}(\Sigma)=I_{V_\Sigma}.
    \end{align*}
    Define
    \begin{align*}
        W_\Sigma:=\frac12\bigl(R_M+|A_\Sigma|^2\bigr),
        \qquad S:=\int_\Sigma W_\Sigma\,\omega_\Sigma.
    \end{align*}
    By \eqref{Equation: Scalar curvature Jacobi decomposition} and
    Gauss--Bonnet,
    \begin{align}\label{Equation: Integrated scalar curvature potential}
        \int_\Sigma V_\Sigma\,\omega_\Sigma
        =S+4\pi\bigl(\gamma(\Sigma)-1\bigr).
    \end{align}
    Also,
    \begin{align*}
        (V_\Sigma)_+\le |A_\Sigma|^2+a\le\max\{1,a/r\}\bigl(|A_\Sigma|^2+R_M\bigr)
        =cW_\Sigma.
    \end{align*}
    Hence Theorem~\ref{Theorem: Scalar spectral inequality} gives,
    for every integer $N\ge2$,
    \begin{align*}
        S+4\pi\bigl(\gamma(\Sigma)-1\bigr)
        \le4\pi\bigl(\gamma(\Sigma)+N-1\bigr)
        +\frac{cI}{N}S.
    \end{align*}
    Cancelling the genus terms, we obtain
    \begin{align}\label{Equation: Scalar curvature absorption estimate}
        S\le4\pi N+\frac{cI}{N}S.
    \end{align}
    
    Suppose first that $I\ge1$. Taking $N=mI\ge2$ and using
    $m\ge2c$, we deduce
    \begin{align*}
        S\le8\pi mI.
    \end{align*}
    Consequently,
    \begin{align}\label{Equation: Positive scalar explicit curvature bound}
        r\operatorname{Area}(\Sigma)
        +\int_\Sigma|A_\Sigma|^2\,\omega_\Sigma
        \le2S\le16\pi mI.
    \end{align}
    Moreover, \eqref{Equation: Integrated scalar curvature potential}
    yields
    \begin{align*}
        4\pi\bigl(\gamma(\Sigma)-1\bigr)=\int_\Sigma V_\Sigma\,\omega_\Sigma-S\le(c-1)S\le8\pi m(c-1)I.
    \end{align*}
    Since $I\ge1$, this implies
    \begin{align*}
        \gamma(\Sigma)\le\bigl(1+2m(c-1)\bigr)I.
    \end{align*}
    
    If $I=0$, as in \cite{SY79,FCS80}, testing the Jacobi form with the constant function $1$ gives
    \begin{align*}
        S+4\pi\bigl(\gamma(\Sigma)-1\bigr)
        =\int_\Sigma V_\Sigma\,\omega_\Sigma\le0.
    \end{align*}
    Since $S>0$, necessarily $\gamma(\Sigma)=0$ and $S\le4\pi$.
    In particular,
    \begin{align}\label{Equation: Positive scalar stable curvature bound}
        \int_\Sigma\bigl(R_M+|A_\Sigma|^2\bigr)\,\omega_\Sigma
        =2S\le8\pi.
    \end{align}
    Combining the two cases proves both estimates, for instance with
    \begin{align*}
        C:=\max\left\{1+2m(c-1),
        \frac{16\pi m}{\min\{1,r\}}\right\}.
    \end{align*}
    This constant depends only on the ambient metric.
\end{proof}
\section{Localization and the general embedded case}
\label{Section: General embedded case}

In this section, $(M^3,\bar g)$ is a fixed closed Riemannian manifold.
We extend the genus--index--area estimate of the preceding section to closed embedded
minimal surfaces with arbitrary tangent and normal bundles. The argument
uses local normal sections to compare the normal Morse index with the
index of a shifted scalar operator on an orientable cover.

Choose $\kappa\ge0$ such that $\operatorname{sec}_M\ge-\kappa$.
Fix a finite cover of $M$ by coordinate balls $U_1,\ldots,U_m$ and a
smooth nonnegative partition of unity $\rho_1,\ldots,\rho_m$ with
$\operatorname{supp}\rho_j\Subset U_j$. Define
\begin{align*}
    \chi_j:=\frac{\rho_j}{\bigl(\sum_{\ell=1}^m\rho_\ell^2\bigr)^{1/2}}.
\end{align*}
The denominator is positive everywhere, so each $\chi_j$ is smooth and
compactly supported in $U_j$. We have
\begin{align}\label{Equation: Ambient localization data}
    \sum_{j=1}^m\chi_j^2=1,
    \qquad
    c_M:=\sup_M\sum_{j=1}^m|\nabla^M\chi_j|_{\bar g}^2<+\infty.
\end{align}
All these choices are fixed independently of the surface. When restricting
to a surface $\Sigma\subset M$, we continue to write $\chi_j$ for
$\chi_j|_\Sigma$. The area measures and the normal quadratic form
$Q_\Sigma^\perp$ are as in Section~\ref{Section: Curvature potentials}.

The following lemma follows from \cite[Lemma C.1]{CKM17} by passing to a coordinate chart and using the equivalence between orientability and two-sidedness for hypersurfaces in an orientable ambient manifold. We include a proof for completeness.

\begin{lemma}[Local orientability and two-sidedness]
\label{Lemma: Local orientability and two sidedness}
    Let $\Sigma\subset M$ be a closed smoothly embedded surface, and
    let $U\subset M$ be a coordinate ball. Every connected component
    of $\Sigma\cap U$ is orientable and two-sided in $U$.
\end{lemma}
\begin{proof}
    Since $\Sigma$ is closed in $M$, its intersection with $U$ is a
    properly embedded surface without boundary in $U$. The same is
    true of each connected component $S$ of $\Sigma\cap U$.

    Suppose that the normal line bundle of $S$ is nontrivial. Then
    there exists a smooth loop in $S$ along which a continuous unit
    normal returns with the opposite sign. A sufficiently small normal
    displacement of this loop gives a path in $U\setminus S$ joining
    the two sides of a normal segment at its base point. Closing this
    path by that normal segment and smoothing away from $S$, we obtain
    a smooth loop $\ell$ in $U$ meeting $S$ transversely at exactly
    one parameter value.

    Since $U$ is simply connected, $\ell$ extends to a smooth map
    $u:\overline{\mathbb{D}}\to U$. By relative transversality, we may
    arrange that $u$ is transverse to $S$ while preserving its boundary
    values. Proper embeddedness of $S$ implies that $u^{-1}(S)$ is
    compact. It is therefore a compact one-dimensional manifold whose
    boundary consists of the single point of $\ell^{-1}(S)$. This
    contradicts the fact that the boundary of a compact one-dimensional
    manifold has even cardinality.

    Hence $S$ is two-sided. Choosing an orientation of the coordinate
    ball $U$, a smooth unit normal on $S$ induces an orientation of
    $S$, proving the claim.
\end{proof}

\begin{proposition}[Localized comparison with a shifted scalar index]
\label{Proposition: Localized shifted index comparison}
    Let $\Sigma\subset M$ be a closed connected smoothly embedded
    minimal surface, and set $I:=\operatorname{Ind}(\Sigma)$.
    If $\Sigma$ is orientable, let $p:\widehat\Sigma\to\Sigma$ be
    the identity map. Otherwise, let $p$ be its connected orientation
    double cover. Set
    \begin{align*}
        q:=\deg p\in\{1,2\},
        \qquad \widehat h:=p^*h,
        \qquad \widehat V:=p^*V_\Sigma.
    \end{align*}
    Then
    \begin{align}\label{Equation: Localized shifted index comparison}
        \operatorname{Ind}(Q_{\widehat V-c_M})\le qmI.
    \end{align}
    More precisely, there exist real linear maps
    \begin{align*}
        T_{j,a}:C^\infty(\widehat\Sigma,\mathbb{R})
        \rightarrow\Gamma^\infty(N\Sigma),
        \qquad 1\le j\le m,\quad 1\le a\le q,
    \end{align*}
    such that, for every $f\in C^\infty(\widehat\Sigma,\mathbb{R})$,
    \begin{align}\label{Equation: Normal localization identity}
        \sum_{j=1}^m\sum_{a=1}^q Q_\Sigma^\perp(T_{j,a}f)=Q_{\widehat V}(f)
        +\int_{\widehat\Sigma}
        p^*\!\left(\sum_{j=1}^m|\nabla^\Sigma\chi_j|^2\right)
        f^2\,\omega_{\widehat\Sigma}\le Q_{\widehat V-c_M}(f).
    \end{align}
\end{proposition}
\begin{proof}
    \medskip
    \noindent
    \textbf{Step 1: construction of the local normal sections.}
    Set $D_j:=\Sigma\cap U_j$. By
    Lemma~\ref{Lemma: Local orientability and two sidedness}, we can
    choose a smooth unit normal $\nu_j$ on $D_j$. Moreover, the
    restriction of $p$ to $p^{-1}(D_j)$ is a trivial cover of degree
    $q$. For $q=1$ this is immediate. For $q=2$, it follows because
    each connected component of $D_j$ is orientable: choosing an
    orientation on that component distinguishes its two lifts to the
    orientation cover.

    We can therefore choose smooth inverse branches
    \begin{align*}
        \sigma_{j,a}:D_j\rightarrow\widehat\Sigma,
        \qquad p\circ\sigma_{j,a}=\operatorname{Id}_{D_j},
        \qquad a=1,\ldots,q,
    \end{align*}
    whose images are disjoint and exhaust $p^{-1}(D_j)$. The choices
    are made separately on the connected components of $D_j$, and
    together define $q$ smooth branches on all of $D_j$.

    Given $f\in C^\infty(\widehat\Sigma,\mathbb{R})$, write
    $f_{j,a}:=f\circ\sigma_{j,a}$ on $D_j$, and define
    \begin{align*}
        (T_{j,a}f)(x):=
        \begin{cases}
            \chi_j(x)f_{j,a}(x)\nu_j(x),&x\in D_j,\\
            0,&x\in\Sigma\setminus D_j.
        \end{cases}
    \end{align*}
    Since $\operatorname{supp}\chi_j\Subset U_j$, this extension by
    zero is a smooth global normal section. The resulting maps
    $T_{j,a}$ are real linear. In particular, there are exactly $q$
    maps for each chart, regardless of the number of connected
    components of $D_j$.

    \medskip
    \noindent
    \textbf{Step 2: the localization identity.}
    Set $\widehat\chi_j:=\chi_j\circ p$. A local unit section of the
    normal line bundle is parallel for the normal connection, so
    $\nabla^\perp\nu_j=0$. Using
    \eqref{Equation: Normal Jacobi quadratic form}, we obtain
    \begin{align*}
        Q_\Sigma^\perp(T_{j,a}f)
        =\int_{D_j}\bigl(
        |\nabla^\Sigma(\chi_jf_{j,a})|^2
        -V_\Sigma\chi_j^2f_{j,a}^2\bigr)\,\omega_\Sigma.
    \end{align*}
    Since $p:(\widehat\Sigma,\widehat h)\to(\Sigma,h)$ is a local
    isometry, changing variables on the $q$ inverse branches gives
    \begin{align}\label{Equation: Localization on one coordinate ball}
        \sum_{a=1}^q Q_\Sigma^\perp(T_{j,a}f)
        =\int_{\widehat\Sigma}\bigl(
        |\nabla^{\widehat\Sigma}(\widehat\chi_jf)|^2
        -\widehat V\widehat\chi_j^2f^2\bigr)\,
        \omega_{\widehat\Sigma}.
    \end{align}
    Here the integrand vanishes outside $p^{-1}(D_j)$ because
    $\widehat\chi_j$ is supported in that set.

    Pulling back \eqref{Equation: Ambient localization data} and
    differentiating, we find
    \begin{align*}
        \sum_{j=1}^m\widehat\chi_j^2=1,
        \qquad
        \sum_{j=1}^m\widehat\chi_j
        \nabla^{\widehat\Sigma}\widehat\chi_j=0.
    \end{align*}
    Expanding the gradient squares therefore yields
    \begin{align*}
        \sum_{j=1}^m
        |\nabla^{\widehat\Sigma}(\widehat\chi_jf)|^2&=\left(\sum_{j=1}^m\widehat\chi_j^2\right)
        |\nabla^{\widehat\Sigma}f|^2+2f\left\langle\nabla^{\widehat\Sigma}f,
        \sum_{j=1}^m\widehat\chi_j
        \nabla^{\widehat\Sigma}\widehat\chi_j\right\rangle+f^2\sum_{j=1}^m
        |\nabla^{\widehat\Sigma}\widehat\chi_j|^2\\
        &=|\nabla^{\widehat\Sigma}f|^2
        +f^2\sum_{j=1}^m
        |\nabla^{\widehat\Sigma}\widehat\chi_j|^2.
    \end{align*}
    Summing \eqref{Equation: Localization on one coordinate ball}
    over $j$ and using $\sum_j\widehat\chi_j^2=1$ in the potential
    term, we obtain
    \begin{align*}
        \sum_{j=1}^m\sum_{a=1}^q Q_\Sigma^\perp(T_{j,a}f)
        =Q_{\widehat V}(f)
        +\int_{\widehat\Sigma}
        \left(\sum_{j=1}^m
        |\nabla^{\widehat\Sigma}\widehat\chi_j|^2\right)
        f^2\,\omega_{\widehat\Sigma}.
    \end{align*}
    Local isometry gives
    \begin{align*}
        \sum_{j=1}^m
        |\nabla^{\widehat\Sigma}\widehat\chi_j|^2
        =p^*\!\left(\sum_{j=1}^m
        |\nabla^\Sigma\chi_j|^2\right).
    \end{align*}
    Furthermore, $\nabla^\Sigma\chi_j$ is the tangential projection
    of $\nabla^M\chi_j$. Hence
    \begin{align*}
        0\le\sum_{j=1}^m|\nabla^\Sigma\chi_j|^2
        \le\sum_{j=1}^m|\nabla^M\chi_j|^2\le c_M.
    \end{align*}
    Since
    \begin{align*}
        Q_{\widehat V-c_M}(f)
        =Q_{\widehat V}(f)
        +c_M\int_{\widehat\Sigma}f^2\,\omega_{\widehat\Sigma},
    \end{align*}
    this proves \eqref{Equation: Normal localization identity}.

    \medskip
    \noindent
    \textbf{Step 3: comparison of the indices.}
    Let $E_-$ be the negative spectral subspace of the self-adjoint
    operator associated with $Q_\Sigma^\perp$, and let
    $\Pi_-:L^2(N\Sigma)\to E_-$ be the orthogonal projection. Then
    $\dim E_-=I$. For every $X\in\Gamma^\infty(N\Sigma)$, the spectral
    theorem gives
    \begin{align*}
        \Pi_-X=0\quad\rightarrow\quad Q_\Sigma^\perp(X)\ge0.
    \end{align*}
    Let $F_-$ be the negative spectral subspace of the scalar operator
    $-\Delta_{\widehat h}-(\widehat V-c_M)$, and write
    $J:=\dim F_-=\operatorname{Ind}(Q_{\widehat V-c_M})$.
    Both negative spectral subspaces are finite-dimensional and consist
    of smooth sections or functions, respectively.

    Consider the real linear map
    \begin{align*}
        \mathcal{A}:F_-\rightarrow
        \bigoplus_{j=1}^m\bigoplus_{a=1}^q E_-,
        \qquad
        \mathcal{A}(f):=
        \bigl(\Pi_-T_{j,a}f\bigr)_{j,a}.
    \end{align*}
    We claim that $\mathcal{A}$ is injective. Indeed, if
    $f\in\ker\mathcal{A}$, then
    $Q_\Sigma^\perp(T_{j,a}f)\ge0$ for every $j,a$. If $f\ne0$,
    negative definiteness on $F_-$ and
    \eqref{Equation: Normal localization identity} would give
    \begin{align*}
        0\le\sum_{j=1}^m\sum_{a=1}^q
        Q_\Sigma^\perp(T_{j,a}f)
        \le Q_{\widehat V-c_M}(f)<0,
    \end{align*}
    a contradiction. Thus $\ker\mathcal{A}=\{0\}$, and comparing
    dimensions yields $J\le qmI$, as required.
\end{proof}

We recall the following standard consequence of the monotonicity
formula, which gives a positive lower bound for the area depending
only on the ambient manifold; see e.g. \cite[Corollary~2.4]{Kok22}.

\begin{lemma}
\label{Lemma: Uniform minimum area}
    There exists $a_M>0$, depending only on $(M,\bar g)$, such that
    every nonempty connected closed minimal immersion
    $\iota:\Sigma\to M$ satisfies
    \begin{align}\label{Equation: Uniform minimum area}
        \operatorname{Area}(\Sigma)\ge a_M.
    \end{align}
\end{lemma}
\begin{proof}
    Compactness of $M$ and the Riemannian monotonicity formula give
    constants $r_M>0$ and $\Lambda_M\ge0$, depending only on
    $(M,\bar g)$, such that for every closed minimal immersion
    $\iota:\Sigma\to M$ and every $x\in\iota(\Sigma)$, the function
    \begin{align*}
        r\mapsto
        e^{\Lambda_M r}
        \frac{\operatorname{Area}(\iota^{-1}(B_r^M(x)))}{r^2}
    \end{align*}
    is nondecreasing on $(0,r_M)$. The formula applies to the area
    measure counted on the domain, without an embeddedness or
    orientability assumption. Since the density at $x$ is at least
    one, we have
    \begin{align*}
        \lim_{r\downarrow0}
        \frac{\operatorname{Area}(\iota^{-1}(B_r^M(x)))}{r^2}
        \ge\pi.
    \end{align*}
    Taking $r=r_M/2$ therefore gives
    \begin{align*}
        \operatorname{Area}(\Sigma)\ge\operatorname{Area}(\iota^{-1}(B_{r_M/2}^M(x)))\ge\pi e^{-\Lambda_M r_M/2}\frac{r_M^2}{4}.
    \end{align*}
    Thus we may take
    $a_M:=\pi e^{-\Lambda_M r_M/2}r_M^2/4>0$.
\end{proof}

\begin{proof}[Proof of Theorem~\ref{Theorem: Topology index and area}]
    Suppose first that $\Sigma$ is connected. Let
    $p:\widehat\Sigma\to\Sigma$, $q$, and $\widehat V$ be as in
    Proposition~\ref{Proposition: Localized shifted index comparison},
    and set
    \begin{align*}
        I:=\operatorname{Ind}(\Sigma),
        \qquad
        J:=\operatorname{Ind}(Q_{\widehat V-c_M}),
        \qquad
        \widehat\gamma:=\gamma(\widehat\Sigma).
    \end{align*}
    The composition of $p$ with the inclusion $\Sigma\hookrightarrow M$
    is a closed connected orientable minimal immersion. Its induced
    metric is $\widehat h=p^*h$, and its scalar Jacobi potential is
    $\widehat V=p^*V_\Sigma$. In particular,
    Corollary~\ref{Corollary: Shifted scalar Jacobi estimate} applies,
    irrespective of the normal bundle of this immersion, and gives
    \begin{align*}
        \widehat\gamma-1
        \le8J+\frac{14\kappa+4c_M}{8\pi}
        \operatorname{Area}(\widehat\Sigma).
    \end{align*}
    By \eqref{Equation: Localized shifted index comparison} and the
    degree of the covering,
    \begin{align*}
        J\le qmI,
        \qquad
        \operatorname{Area}(\widehat\Sigma)
        =q\operatorname{Area}(\Sigma).
    \end{align*}
    Consequently,
    \begin{align*}
        \widehat\gamma-1
        \le8qmI+\frac{14\kappa+4c_M}{8\pi}\,
        q\operatorname{Area}(\Sigma).
    \end{align*}
    Euler characteristic is multiplicative under finite coverings, so
    \begin{align*}
        2-2\widehat\gamma
        =\chi(\widehat\Sigma)=q\chi(\Sigma).
    \end{align*}
    Multiplying the preceding genus estimate by $2/q$, we obtain
    \begin{align}\label{Equation: Embedded Euler characteristic estimate}
        -\chi(\Sigma)
        \le16m\operatorname{Ind}(\Sigma)
        +\frac{14\kappa+4c_M}{4\pi}\operatorname{Area}(\Sigma).
    \end{align}

    For a connected closed surface, $b_0(\Sigma;\mathbb{Z}_2)=1$,
    and Poincar\'e duality over $\mathbb{Z}_2$ gives
    $b_2(\Sigma;\mathbb{Z}_2)=1$. Thus
    \begin{align*}
        b_1(\Sigma;\mathbb{Z}_2)=2-\chi(\Sigma).
    \end{align*}
    Using Lemma~\ref{Lemma: Uniform minimum area} to absorb the
    constant $2$, \eqref{Equation: Embedded Euler characteristic estimate}
    implies
    \begin{align}\label{Equation: Embedded Betti estimate explicit constants}
        b_1(\Sigma;\mathbb{Z}_2)
        \le16m\operatorname{Ind}(\Sigma)
        +\left(\frac{14\kappa+4c_M}{4\pi}+\frac2{a_M}\right)
        \operatorname{Area}(\Sigma).
    \end{align}
    In particular, the required inequality holds for connected surfaces
    with
    \begin{align}\label{Equation: One choice of ambient constant}
        C(M,\bar g):=
        \max\left\{16m,
        \frac{14\kappa+4c_M}{4\pi}+\frac2{a_M}\right\}.
    \end{align}
    This constant depends only on the fixed ambient manifold and the
    choices made at the beginning of the section.

    Finally, let $\Sigma$ be disconnected. Compactness implies that
    it has finitely many connected components
    $\Sigma_1,\ldots,\Sigma_k$. The normal Jacobi form is the direct
    sum of the forms on these components, so
    \begin{align*}
        \operatorname{Ind}(\Sigma)
        =\sum_{\ell=1}^k\operatorname{Ind}(\Sigma_\ell).
    \end{align*}
    Area and the first Betti number are also additive. Applying the
    connected case to each component therefore gives
    \begin{align*}
        b_1(\Sigma;\mathbb{Z}_2)=\sum_{\ell=1}^k b_1(\Sigma_\ell;\mathbb{Z}_2)\le C(M,\bar g)\sum_{\ell=1}^k
        \bigl(\operatorname{Ind}(\Sigma_\ell)
        +\operatorname{Area}(\Sigma_\ell)\bigr)=C(M,\bar g)\bigl(\operatorname{Ind}(\Sigma)
        +\operatorname{Area}(\Sigma)\bigr),
    \end{align*}
    completing the proof.
\end{proof}
%
\begingroup
\linespread{1.08}\selectfont
\bibliographystyle{amsalpha}
\bibliography{main}

\providecommand{\bysame}{\leavevmode\hbox to3em{\hrulefill}\thinspace}
\providecommand{\MR}{\relax\ifhmode\unskip\space\fi MR }
\providecommand{\MRhref}[2]{%
  \href{http://www.ams.org/mathscinet-getitem?mr=#1}{#2}
}
\providecommand{\href}[2]{#2}
\begin{thebibliography}{DLGR25}

\bibitem[ACS18]{ACS}
Lucas Ambrozio, Alessandro Carlotto, and Ben Sharp, \emph{Comparing the {Morse} index and the first {Betti} number of minimal hypersurfaces}, J. Differential Geom. \textbf{108} (2018), no.~3, 379--410.

\bibitem[AGL07]{AGL}
Claudio Arezzo, Alessandro Ghigi, and Andrea Loi, \emph{Stable bundles and the first eigenvalue of the {Laplacian}}, J. Geom. Anal. \textbf{17} (2007), no.~3, 375--386.

\bibitem[BG13]{BG13}
Leonardo Biliotti and Alessandro Ghigi, \emph{{Satake--Furstenberg} compactifications, the moment map and {$\lambda_1$}}, Amer. J. Math. \textbf{135} (2013), no.~1, 237--274.

\bibitem[BLY94]{BLY}
Jean-Pierre Bourguignon, Peter Li, and Shing-Tung Yau, \emph{Upper bound for the first eigenvalue of algebraic submanifolds}, Comment. Math. Helv. \textbf{69} (1994), 199--207.

\bibitem[BS18]{BS18}
Reto Buzano and Ben Sharp, \emph{Qualitative and quantitative estimates for minimal hypersurfaces with bounded index and area}, Trans. Amer. Math. Soc. \textbf{370} (2018), no.~6, 4373--4399.

\bibitem[CG26]{CG26}
Otis Chodosh and Matilde Gianocca, \emph{Minimal hypersurfaces of {Morse} index one}, 2026, Preprint, arXiv:2607.27444.

\bibitem[Chu25]{Chu25}
Adrian Chun-Pong Chu, \emph{Minimal surfaces with arbitrary genus in 3-spheres of positive {Ricci} curvature}, 2025, arXiv:2508.06019.

\bibitem[CKM17]{CKM17}
Otis Chodosh, Daniel Ketover, and Davi Maximo, \emph{Minimal hypersurfaces with bounded index}, Invent. Math. \textbf{209} (2017), no.~3, 617--664.

\bibitem[CM16]{CM16}
Otis Chodosh and Davi Maximo, \emph{On the topology and index of minimal surfaces}, J. Differential Geom. \textbf{104} (2016), no.~3, 399--418.

\bibitem[CM20]{CMant20}
Otis Chodosh and Christos Mantoulidis, \emph{Minimal surfaces and the {Allen--Cahn} equation on $3$-manifolds: index, multiplicity, and curvature estimates}, Ann. of Math. (2) \textbf{191} (2020), no.~1, 213--328.

\bibitem[CM23]{CM23}
Otis Chodosh and Davi Maximo, \emph{On the topology and index of minimal surfaces {II}}, J. Differential Geom. \textbf{123} (2023), no.~3, 431--459.

\bibitem[CMF24]{CF24}
Santiago Cordero-Misteli and Giada Franz, \emph{Estimating the {Morse} index of free boundary minimal hypersurfaces through covering arguments}, J. Reine Angew. Math. \textbf{807} (2024), 187--220.

\bibitem[CS85]{CS85}
Hyeong~In Choi and Richard Schoen, \emph{The space of minimal embeddings of a surface into a three-dimensional manifold of positive {Ricci} curvature}, Invent. Math. \textbf{81} (1985), no.~3, 387--394.

\bibitem[CS25]{ChuStern25}
Adrian Chun-Pong Chu and Daniel Stern, \emph{Minimal surface doublings and electrostatics for {Schr{\"o}dinger} operators}, 2025, arXiv:2509.18630.

\bibitem[CW83]{CW83}
Hyeong~In Choi and Ai-Nung Wang, \emph{A first eigenvalue estimate for minimal hypersurfaces}, J. Differential Geom. \textbf{18} (1983), no.~3, 559--562.

\bibitem[DLGR25]{DGR25}
Francesca Da~Lio, Matilde Gianocca, and Tristan Rivi{\`e}re, \emph{{Morse} index stability for critical points to conformally invariant {Lagrangians}}, J. Eur. Math. Soc. (2025), Published online first.

\bibitem[EM08]{EM}
Norio Ejiri and Mario Micallef, \emph{Comparison between second variation of area and second variation of energy of a minimal surface}, Adv. Calc. Var. \textbf{1} (2008), no.~3, 223--239.

\bibitem[FCS80]{FCS80}
Doris Fischer-Colbrie and Richard Schoen, \emph{The structure of complete stable minimal surfaces in {$3$}-manifolds of non-negative scalar curvature}, Comm. Pure Appl. Math. \textbf{33} (1980), no.~2, 199--211.

\bibitem[GNS16]{GNS16}
Alexander Grigor'yan, Nikolai Nadirashvili, and Yannick Sire, \emph{A lower bound for the number of negative eigenvalues of {Schr{\"o}dinger} operators}, Journal of Differential Geometry \textbf{102} (2016), no.~3, 395--408.

\bibitem[Has11]{Has11}
Asma Hassannezhad, \emph{Conformal upper bounds for the eigenvalues of the {L}aplacian and {S}teklov problem}, J. Funct. Anal. \textbf{261} (2011), no.~12, 3419--3436.

\bibitem[Has13]{Hass13}
\bysame, \emph{Eigenvalues of perturbed {L}aplace operators on compact manifolds}, Pacific J. Math. \textbf{264} (2013), no.~2, 333--354.

\bibitem[Her70]{Hersch}
Joseph Hersch, \emph{Quatre propri{\'e}t{\'e}s isop{\'e}rim{\'e}triques de membranes sph{\'e}riques homog{\`e}nes}, C. R. Acad. Sci. Paris S{\'e}r. A-B \textbf{270} (1970), A1645--A1648.

\bibitem[IMN18]{IMN18}
Kei Irie, Fernando~C. Marques, and Andr{\'e} Neves, \emph{Density of minimal hypersurfaces for generic metrics}, Ann. of Math. (2) \textbf{187} (2018), no.~3, 963--972.

\bibitem[Kok20]{Kok20}
Gerasim Kokarev, \emph{Bounds for {L}aplace eigenvalues of {K}{\"a}hler metrics}, Adv. Math. \textbf{365} (2020), 107061.

\bibitem[Kok22]{Kok22}
\bysame, \emph{{Berger}'s inequality in the presence of upper sectional curvature bound}, International Mathematics Research Notices \textbf{2022} (2022), no.~15, 11262--11303.

\bibitem[Kor93]{Korevaar}
Nicholas Korevaar, \emph{Upper bounds for eigenvalues of conformal metrics}, J. Differential Geom. \textbf{37} (1993), no.~1, 73--93.

\bibitem[KV22]{KV22}
Mikhail Karpukhin and Denis Vinokurov, \emph{The first eigenvalue of the {Laplacian} on orientable surfaces}, Math. Z. \textbf{301} (2022), no.~3, 2733--2746.

\bibitem[LMN18]{LMN18}
Yevgeny Liokumovich, Fernando~C. Marques, and Andr{\'e} Neves, \emph{Weyl law for the volume spectrum}, Ann. of Math. (2) \textbf{187} (2018), no.~3, 933--961.

\bibitem[LY82]{LY}
Peter Li and Shing-Tung Yau, \emph{A new conformal invariant and its applications to the {W}illmore conjecture and the first eigenvalue of compact surfaces}, Invent. Math. \textbf{69} (1982), 269--291.

\bibitem[Mar14]{MarquesICM14}
Fernando~Cod{\'a} Marques, \emph{Minimal surfaces: variational theory and applications}, Proceedings of the International Congress of Mathematicians---Seoul 2014, vol.~I, Kyung Moon Sa, Seoul, 2014, pp.~283--310.

\bibitem[Max23]{Maximo23}
Davi Maximo, \emph{A note on minimal surfaces with bounded index}, Comm. Anal. Geom. \textbf{31} (2023), no.~5, 1069--1078.

\bibitem[MN17]{MN17}
Fernando~C. Marques and Andr{\'e} Neves, \emph{Existence of infinitely many minimal hypersurfaces in positive {Ricci} curvature}, Invent. Math. \textbf{209} (2017), no.~2, 577--616.

\bibitem[MN21]{MN21}
\bysame, \emph{Morse index of multiplicity one min-max minimal hypersurfaces}, Adv. Math. \textbf{378} (2021), 107527.

\bibitem[MPR19]{MPR19}
William~H. Meeks, III, Joaqu{\'i}n P{\'e}rez, and Antonio Ros, \emph{Bounds on the topology and index of minimal surfaces}, Acta Math. \textbf{223} (2019), no.~1, 113--149.

\bibitem[MRS26]{MRS26}
Davi Maximo, Philipp Reiser, and Daniele Semola, \emph{{Ricci} curvature and minimal hypersurfaces with large {Betti} numbers}, Calc. Var. Partial Differential Equations \textbf{65} (2026), Paper No. 120.

\bibitem[Nev14]{NevesICM14}
Andr{\'e} Neves, \emph{New applications of min-max theory}, Proceedings of the International Congress of Mathematicians---Seoul 2014, vol.~II, Kyung Moon Sa, Seoul, 2014, pp.~939--957.

\bibitem[Pit81]{Pitts81}
Jon~T. Pitts, \emph{Existence and regularity of minimal surfaces on {Riemannian} manifolds}, Mathematical Notes, vol.~27, Princeton University Press, Princeton, N.J., 1981.

\bibitem[Ros06]{Ros06}
Antonio Ros, \emph{One-sided complete stable minimal surfaces}, J. Differential Geom. \textbf{74} (2006), no.~1, 69--92.

\bibitem[Ros22]{RosGenusThree}
\bysame, \emph{On the first eigenvalue of the {Laplacian} on compact surfaces of genus three}, J. Math. Soc. Japan \textbf{74} (2022), no.~3, 813--828.

\bibitem[Ros23]{RosLargeGenus}
\bysame, \emph{First eigenvalue of the {Laplacian} on compact surfaces for large genera}, Math. Z. \textbf{305} (2023), no.~4, Paper No. 62.

\bibitem[Ros25]{Ros25}
\bysame, \emph{Index one minimal surfaces in positively curved {$3$}-manifolds}, J. Differential Geom. \textbf{129} (2025), no.~2, 523--540.

\bibitem[Sav10]{Savo10}
Alessandro Savo, \emph{Index bounds for minimal hypersurfaces of the sphere}, Indiana Univ. Math. J. \textbf{59} (2010), no.~3, 823--837.

\bibitem[Sha17]{Sharp17}
Ben Sharp, \emph{Compactness of minimal hypersurfaces with bounded index}, J. Differential Geom. \textbf{106} (2017), no.~2, 317--339.

\bibitem[Son23a]{SongYau23}
Antoine Song, \emph{Existence of infinitely many minimal hypersurfaces in closed manifolds}, Ann. of Math. (2) \textbf{197} (2023), no.~3, 859--895.

\bibitem[Son23b]{Song}
\bysame, \emph{{Morse} index, {Betti} numbers, and singular set of bounded area minimal hypersurfaces}, Duke Math. J. \textbf{172} (2023), no.~11, 2073--2147.

\bibitem[SS81]{SS81}
Richard Schoen and Leon Simon, \emph{Regularity of stable minimal hypersurfaces}, Comm. Pure Appl. Math. \textbf{34} (1981), no.~6, 741--797.

\bibitem[SY79]{SY79}
Richard Schoen and Shing-Tung Yau, \emph{Existence of incompressible minimal surfaces and the topology of three dimensional manifolds with non-negative scalar curvature}, Ann. of Math. (2) \textbf{110} (1979), no.~1, 127--142.

\bibitem[YY80]{YangYau}
Paul~C. Yang and Shing-Tung Yau, \emph{Eigenvalues of the {L}aplacian of compact {R}iemann surfaces and minimal submanifolds}, Ann. Scuola Norm. Sup. Pisa Cl. Sci. (4) \textbf{7} (1980), no.~1, 55--63.

\bibitem[Zho20]{Zhou20}
Xin Zhou, \emph{On the {Multiplicity One Conjecture} in min-max theory}, Ann. of Math. (2) \textbf{192} (2020), no.~3, 767--820.

\end{thebibliography}
\endgroup
\end{document}